\documentclass[12pt]{amsart}

\usepackage{times, amsmath, amsfonts, amssymb, amsthm}
\usepackage{times}
\usepackage{natbib}
\usepackage{srcltx}
\usepackage{setspace,enumerate}

\newcommand{\cA}{\mathcal{A}}

\newcommand{\cC}{\mathcal{C}}
\newcommand{\cD}{\mathcal{D}}

\newcommand{\cR}{\mathcal{R}}

\newcommand{\eps}{\varepsilon}
\newcommand{\ve}{\varepsilon}

\newcommand{\wt}{\widetilde}
\DeclareMathOperator{\erfi}{erfi}
\DeclareMathOperator{\erf}{erf}

\theoremstyle{plain}
\newtheorem{theorem}{Theorem}[section]

\newtheorem{proposition}[theorem]{Proposition}

\theoremstyle{definition}
\newtheorem{remark}[theorem]{Remark}

\newtheorem{defi}[theorem]{Definition}

\newtheorem{example}[theorem]{Example}

\newtheorem{problem}[theorem]{Problem}

\renewenvironment{proof}[1][] {\noindent {\bf Proof#1.} }{\hspace*{\fill}$\square$\medskip\par}

\newcommand{\N}{{\mathbb N}}
\newcommand{\Z}{{\mathbb Z}}
\newcommand{\R}{{\mathbb R}}

\newcommand{\E}{{\mathbb E}}
\newcommand{\bs}{{\mathbf s}}
\renewcommand{\P}{{\mathbb P}}
\newcommand{\F}{{\mathcal F}}

\renewcommand{\P}{{\mathbb P}}

\newcommand{\ol}{\overline}

 \numberwithin{equation}{section}

\author{Krzysztof Burdzy
\and Michael Scheutzow}

\address{KB: Department of Mathematics, Box 354350,
University of Washington, Seattle, WA 98195, USA}
\email{burdzy@math.washington.edu}
\address{MS: Institut f\"ur Mathematik, MA 7-5, Fakult\"at II, 
Technische Universit\"at Berlin, 
Stra\ss e des 17.~Juni 136, 10623 Berlin, Germany }
\email{ms@math.tu-berlin.de}

 \thanks{
K.~Burdzy's research was supported in part by NSF Grant DMS-1206276. The authors are grateful for the hospitality and support of Centre Interfacultaire Bernoulli, Ecole Polytechnique F\'ed\'erale de Lausanne,
where part of this research was done.
}

\title{Forward Brownian Motion}

\date{}

\begin{document}

\begin{abstract}\noindent
We consider processes which have the distribution of standard Brownian motion (in the forward direction of time) starting from random points on the trajectory which accumulate at $-\infty$. We show that these processes 
do not have to have the distribution of standard Brownian motion in the backward direction of time, no matter  which random time we take as the origin. We study the maximum and minimum rates of growth for these processes in the backward direction. We also address the question of which extra assumptions make one of these processes a two-sided Brownian motion.
\end{abstract}

\maketitle

\section{Introduction}

This article is devoted to {\em forward Brownian motions}, i.e., processes defined on the whole real line which appear to be Brownian motions when observed from random points in space-time in the forward time direction.  
More precisely, we will say that $\{X_t,\, t \in \R\}$ is a {\em forward Brownian motion} (FBM) if there exists a sequence $\{S_n,n \leq 0\}$ of 
random times such that $\lim_{n \to -\infty} S_n=-\infty$, a.s., and for every $n$, the process 
$\{X_{S_n+t}-X_{S_n},\,t \ge 0\}$ is standard Brownian motion on $[0,\infty)$. 

A simple example of FBM is 
two-sided Brownian motion, i.e., the process
$\{X_t,\, t \in \R\}$ such that 
$\{X_{t},\,t \ge 0\}$ and  $\{X_{-t},\,t \ge 0\}$ are independent standard Brownian motions. 

We will address several families of questions. It is natural to start with the very general question of whether there are any forward Brownian motions that are significantly different from 
two-sided Brownian motion? The question is somewhat vague but we believe that the answer is quite clear. We will exhibit a number of FBM's that are very different from two-sided Brownian motion by any measure.

We will say that
$\{X_t,\, t \in \R\}$ is  {\em backward Brownian motion}  if $\{X_{-t},\, t \in \R\}$ is FBM. If a process is both a forward Brownian motion and a backward Brownian motion, is it necessarily two-sided Brownian motion (or a very similar process)? The answer is no --- we will present an example to this effect.

It is easy to see that some FBM's can be constructed by concatenating pieces of independent standard Brownian motions. We will show that the family of FBM's  constructed in this way is very rich. One could hope that every FBM may be represented this way --- that would provide a convenient technical tool. Somewhat disappointingly, this turns out to be false. This leaves open the question of characterizing all FBM's. This problem is vague but we can indicate what we mean by invoking well known ``characterizations'' of some families of stochastic processes. L\'evy processes are characterized by the L\'evy-Khinchin exponent; Gaussian processes are characterized by the mean and covariance functions; one-dimensional diffusions are characterized by the scale function and speed measure. So far, we have not found a similar characterization for FBM's.

We will start our rigorous study of FBM's by presenting several results on their path behavior. We will show that FBM trajectories can be extremely different from those of two-sided Brownian motion.

The paper has two disparate sources of inspiration. On the technical side, FBM's arise naturally in the study of skew-Brownian motion (we will be more specific below). On the philosophical 
and scientific side, one may ask what can be said about a stochastic process representing a natural phenomenon which can be observed effectively only in one direction of time. 
If it appears to be Brownian motion, does it necessarily imply that the trajectories of this process have to be Brownian-like in the opposite direction of time? 
The motivation for this question is provided by processes that occur on a scale that is borderline feasible for effective observations, such as the evolution of species.

Our paper is related to a number of models and problems that appeared in literature. ``Extended chains'' were discussed in \cite[Chap.~10]{KSK} (see especially Definition 10-5). A duality problem for L\'evy processes was studied in \cite{BS}; our Example \ref{d17.1} is a special case of that model. Forward Brownian motion is also related to Brownian interlacements that were presented in \cite{SZ}. 

The rest of the paper is organized as follows. Section \ref{sec:def} presents basic definitions and examples. The minimum and maximum rates of growth of FBM trajectories in the 
backward direction are studied in Sections \ref{maxrange} and \ref{minrange}. We analyze the question of which extra assumptions make a decomposable FBM a 2-sided Brownian motion 
in Section \ref{sec:2BM}. We show that a process that is simultaneously a forward Brownian motion and a backward Brownian motion is not necessarily a 2-sided Brownian motion 
in Section \ref{sec:FBMBBM}. Finally, we list some open problems in Section \ref{sec:open}.

\section{Definitions and basic examples}\label{sec:def}

The sets of real numbers, non-negative natural numbers, strictly positive natural numbers and integers will be denoted $\R, \N_0, \N$ and $\Z$, respectively.

Unless stated otherwise, the terms {\em standard Brownian motion} and {\em Brownian motion} will be treated as synonyms and we will assume that these processes start at 0 at time 0.

\subsection{Definitions}

\begin{defi}\label{d17.4}
We will say that $\{X_t,\, t \in \R\}$ is a {\em forward Brownian motion} (FBM) if there exists a sequence $\{S_n,\,-n \in \N_0\}$ of 
random times such that $\lim_{n \to -\infty} S_n=-\infty$, a.s., and for every $n$, the process 
$\{X_{S_n+t}-X_{S_n},\,t \ge 0\}$ is  standard Brownian motion on $[0,\infty)$. We will say that
$\{X_t,\, t \in \R\}$ is  {\em backward Brownian motion} (BBM) if $\{X_{-t},\, t \in \R\}$ is FBM. Further, we 
call $\{X_t,\, t \in \R\}$ {\em two-sided Brownian motion} (2BM) if there exists a random time $S$ such that 
$\{X_{S+t}-X_S,\,t \ge 0\}$ and  $\{X_{S-t}-X_S,\,t \ge 0\}$ are independent standard Brownian motions. 
If we can take $S\equiv 0$ in the last definition then the distribution of $X$ will be denoted 2BM(0).
\end{defi}

Note that the formal definition of two-sided Brownian motion given above is less restrictive than the informal definition given in the introduction.

If $X$ is 2BM then it is FBM. To see this, let $S_n = S+n$, $-n\in \N_0$, in Definition \ref{d17.4}. Similarly, if $X$ is 2BM then it is BBM.

We will describe a general method of constructing forward Brownian motions.
\begin{defi}\label{d17.3}
For each $k\in \Z  $, let
$\{B^k_t, t\geq 0\}$ be a Brownian motion with respect to some normal filtration $\{\F_t^k,t \ge 0\}$ and let $T_k$ be a stopping time with 
respect to $\F^k$. Assume that $(T_k, \{B^k_t, t\in[0,T_k]\})$, $k\in \Z $ are independent and, a.s., $0 \leq T_k < \infty$, for $k\in \Z $, 
$\sum_{k=1}^\infty T_k = \infty$ and $\sum_{k=-\infty}^{-1} T_k = \infty$.
Let $S_0=0$, and note that the conditions $S_{k+1} - S_{k} = T_k$, $k\in\Z$, define uniquely $S_k$ for all $k\in \Z$.
Let $X$ be the unique continuous process such that $X_0 =0$ and
$X_{S_k +t} - X_{S_k} = B^{k}_{t}$ for $t\in [0, T_{k})$, $k \in\Z$.
If a process $X'$ is such that for some random time $U$, the process $\{X_t := X'_{U+t}- X'_U,t\in\R\}$ can be represented as above, then we will call $X'$ \emph{decomposable}.
If   
$(T_k, \{B^k_t, t\in[0,T_k]\})$, $k\in \Z $, 
are i.i.d.~then $X'$ will be called \emph{strongly decomposable}.
\end{defi}

A decomposable process is the concatenation of a countable number of independent (but not necessarily identically distributed) pieces of Brownian trajectories. It follows from the strong Markov property that $\{X_{S_n+t} - X_{S_n}, t\geq 0\}$ is standard Brownian motion for every $n\in \Z$. Hence, a decomposable process is FBM. 

\begin{remark}\label{d18.5}
Recall the condition $\lim_{n \to -\infty} S_n=-\infty$ that appears in Definition \ref{d17.4}. The following list contains this condition and its alternatives (all conditions are supposed to hold a.s.).
\begin{enumerate}[(i)]
\item $\lim_{n \to -\infty} S_n=-\infty$.
\item $\lim_{n \to -\infty} S_n=-\infty$ and $S_n \leq S_{n+1}$ for all $-n\in \N$.
\item $\lim_{n \to -\infty} S_n=-\infty$ and $S_{n+1}$ is a stopping time relative to the filtration generated by $\{X_{S_n+t}-X_{S_n},\,t \ge 0\}$, for all $-n\in\N$.
\item $\liminf_{n \to -\infty} S_n=-\infty$.
\end{enumerate}
Clearly a process satisfying (iii) satisfies (ii), and (ii) is stronger than (i). 
On the other hand, (iv) is weaker than (i).
It is easy to see that (iii) is equivalent to the process $X$ being decomposable.

It is natural to ask if all conditions are in fact equivalent. Proposition \ref{d22.1}~(ii) shows that not all FBM's are decomposable, so (iii) is not equivalent to (i). 

The equivalence of (i) and (ii) would be proved if we could show that if $\{X_t, \,t \in \R\}$ is a process and $S$ and $T$ are random times such that both $\{X_{S+t}-X_S, t\geq 0\}$ and 
$\{X_{T+t}-X_T, t\ge 0\}$ are Brownian motions, then $\{X_{(S \wedge T)+t}-X_{S\wedge T}, \, t \ge 0\}$ 
is Brownian motion. This is not true---not even if we assume that $X$ is two sided Brownian motion.
As an example, let $X$ be two sided Brownian motion with $X_0=0$, $S \equiv 0$ and 
$\Omega_0:=\{\omega: X_1(\omega)>0\}$. Let 
$N=\sup \{n \in \N: X_{-k+1}-X_{-k} >0 \mbox{ for all } k \in \{1,...,n\}\}$ with the convention $\sup \emptyset = 0$. 
Define $T$ to be 1 on $\Omega_0$ and $-N$ otherwise. It is easy to check that both 
$X_{S+t}-X_S$ and $X_{T+t}-X_T$ are Brownian motions but $X_{(S \wedge T)+t}-X_{S\wedge T}$ is not (nor is 
$X_{(S \vee T)+t}-X_{S\vee T}$).
\end{remark}

\begin{remark}\label{d18.6}
It is not true that for every FBM $X$ there exists a random time $T$ such that $\{X_{T+t}-X_T, t\geq 0\}$ and $\{X_{T-t}-X_T, t\geq 0\}$ are independent and $\{X_{T+t}-X_T, t\geq 0\}$ is standard Brownian motion. A counterexample is given in Proposition \ref{d22.1}~(i).
See Example \ref{d18.11} for a weaker, but much easier to prove, claim.
\end{remark}

Sometimes it will be convenient to work with the discrete version of FBM, i.e., forward random walk defined as follows. 

\begin{defi}\label{d18.7}
We will say that an integer valued process $\{Z_n,\, n \in \Z\}$ is {\em forward random walk} (FRW) if there exists a sequence $\{S_n,\,-n \in \N_0\}$ of integer valued
random times such that $\lim_{n \to -\infty} S_n=-\infty$, a.s., and for every $n$, the process 
$\{Z_{S_n+k}-Z_{S_n},\,k \in \N_0\}$ is simple symmetric random walk. We will say that
$\{Z_n,\, n \in \Z\}$ is  {\em backward random walk} (BRW) if $\{Z_{-n},\, n \in \Z\}$ is FRW. We 
call $\{Z_n,\, n \in \Z\}$ {\em two-sided random walk} (2RW) if there exists a random time $S$ such that 
$\{Z_{S+k}-Z_S,\,k\in \N_0\}$ and $\{Z_{S-k}-Z_S,\,k\in \N_0\}$ are independent simple symmetric random walks. 
If we can take $S\equiv 0$ in the last definition then the distribution of $Z$ will be denoted 2RW(0).
\end{defi}

Decomposable and strongly decomposable FRW are defined in a way analogous to that in Definition \ref{d17.3}.

\begin{remark}\label{d18.10}
We will now discuss the relationship between forward Brownian motion and forward random walk.

(i) Suppose that $\{Z_k, k\in\Z\}$ is an integer valued process with the property that $|Z_{k+1} - Z_k| = 1$ for all $k\in \Z$, a.s. We do not assume that the distribution of $Z$ is that of random walk but our construction will be easiest to understand if one keeps in mind a particular example, namely, that of 2RW(0).

Let $\{U_t, t\in [0, \tau_U]\}$ be one dimensional Brownian motion starting at 0, conditioned to stay positive and stopped at the hitting time of 1. In other words, $U$ is Doob's $h$-process in $[0,1]$, starting from 0 and conditioned to converge to 1. Yet another way to think about $U$ is that it is 3-dimensional Bessel process stopped at the hitting time of 1. The process $\{U_{\tau_U } - U_{\tau_U-t},  t\in [0, \tau_U]\}$ has the same distribution as $\{U_t, t\in [0, \tau_U]\}$. See \cite{Will74} for a justification of these claims.

Suppose that $\{B_t, t\geq 0\}$ is standard Brownian motion starting from 0 and let $\tau_B = \inf\{t\geq 0: |B_t| = 1\}$ and $\sigma_B = \sup\{t\leq \tau_B: B_t = 0\}$. Excursion theory easily shows that the processes 
$\{B_t, t\in[0, \sigma_B]\}$ and 
$\{B_{\sigma_B-t}, t\in[0, \sigma_B]\}$
have the same distribution.

Let $\{U^k_t, t\in [0, \tau^k_U]\}$, $k\in\Z$, be i.i.d.{} copies of $\{U_t, t\in [0, \tau_U]\}$ and let $\{B^k_t, t\in[0, \sigma^k_B]\}$, $k\in\Z$, be i.i.d.{} copies of $\{B_t, t\in[0, \sigma_B]\}$ (also independent of $U^k$'s). We will write $U^k_t = U^k(t)$ for typesetting reasons, and similarly for other processes.
Let $M_0 =0$ and define $M_j$ for $j\in\Z$ by 
$M_{j+1} - M_j = \sigma^j_B + \tau^j_U$. For $t\in\R$, a.s., there exists a unique $j\in\Z$ such that 
$M_j \leq t < M_{j+1}$. 
For such $t$ and $j$, let
\begin{align*}
\begin{cases}
X_t = Z_j + B^j(t- M_j) & \text{for  } t < M_j +\sigma^{j}_B,\\
X_t = Z_j + (Z_{j+1}-Z_j)U^j(t- M_j- \sigma^{j}_B)
& \text{otherwise}.
\end{cases}
\end{align*}

It is routine to check that if, for a random integer valued time $S$, 
$\{Z_{S+k}-Z_{S},\,k \in \N_0\}$ is simple symmetric random walk then 
$\{X_{S+t}-X_{S},\,t\geq 0\}$ is Brownian motion. Hence, if $Z$ is FRW then $X$ is FBM.
Moreover, time reversibility of the processes used in the construction explained above implies that if, for a random integer valued time $S$, 
$\{Z_{S-k}-Z_{S},\,k \in \N_0\}$ is simple symmetric random walk then 
$\{X(S+\sigma^S_B-t)-X(S+\sigma^S_B),\,t\geq 0\}$
is Brownian motion. It follows that 
if $Z$ is BRW then $X$ is BBM. 

\smallskip

(ii) We will present a relationship between 2BM and 2RW that goes in the opposite direction, i.e., we will define 2RW starting with 2BM.

Suppose that $X$ is 2BM and $S$ is such that 
$\{X_{S+t}-X_S,\,t \ge 0\}$ and  $\{X_{S-t}-X_S,\,t \ge 0\}$ are independent standard Brownian motions. 

Let $U_0:=S$ 
and for $k\geq 1$, let $U_k := \inf\{t\geq U_{k-1}: |X_t - X_{U_{k-1}}| =1\}$. For $k\in \Z$, $k < 0$, we let 
$U_k := \sup\{t\leq U_{k+1}: |X_t - X_{U_{k+1}}| =1\}$. Let $Z_k = X_{U_k}-X_S$ for $k\in \Z$.

It follows from the strong Markov property of 2BM that 
$Z$ is 2RW(0).
\end{remark}

\subsection{Basic examples}
We start with elementary examples of FBM's.

\begin{example} \label{ex1}
(i) Suppose that $X$ is FBM and recall the random times $S_n$ in Definition \ref{d17.4}. Suppose that for every $n\in \N_0$, there exists a (non-random) real number $s_n$ such that $S_n = s_n$, a.s. Then $X$ is 2BM. To see this, note that since the $S_n$'s are deterministic, the finite dimensional distributions of $\{X_t - X_0, t\in \R\}$ are Gaussian with mean equal to 0 and the same covariance function as for Brownian motion.

(ii) 
A slightly more general example than that in part (i) is the following. We will use the notation of Definition \ref{d17.3}. Suppose that there exists a sequence of (non-random) real numbers $t_k>0$ such that 
$\sum_{k=1}^\infty t_k = \infty$ and $\sum_{k=-\infty}^{-1} t_k = \infty$, and 
$T_k \equiv t_k$ for all $k$. If $X$ is a decomposable FBM corresponding to the $T_k$'s then $X$ is 2BM.
\end{example}

The following example is the starting point of our project, in a sense.
We will construct an FBM which is not two-sided Brownian motion. We will also introduce an idea that will be the basis of a number of our arguments. The example is a special case of duality relationship studied in \cite{BS}.

\begin{example}\label{d17.1}
Recall the notation from Definition  \ref{d17.3}. Suppose that $X$ is strongly decomposable, $X_0=S_0=0$ and $T_k = \inf\{t\geq 0: B^k_t = -1\}$ for all $k$. Note that $X_t \geq -k$ for all $t\leq S_k$, $k\in \Z$, a.s. It follows that $\lim_{t\to -\infty} X_t = \infty$, a.s. Hence, $X_t$ is not backward Brownian motion and, therefore, it is not two-sided Brownian motion. 

We will show that $\{X_{-t}, t\geq 0\}$ is 3-dimensional Bessel process. By the strong Markov property, for $k<0$, the process $\{X_{S_k +t } - X_{S_k}, t\in[0, -S_k]\}$ is Brownian motion stopped at the first hitting time of $k$.
By \cite[Thm.~3.4]{Will74}, the time reversed process $\{X_{-t}, t\in [0, -S_k]\}$ is 3-dimensional Bessel process stopped at the last exit time from $-k$. Since $k$ is arbitrary and $S_k\to -\infty$, a.s., we conclude that 
$\{X_{-t}, t\in [0, \infty)\}$ is 3-dimensional Bessel process.
\end{example}

The following example provided the original motivation
for this project. In a sense, it is a generalization of Example \ref{d17.1}.

\begin{example}\label{d17.2}
Given a standard Brownian motion $B$ and $- 1 \leq \beta \leq 1$, the equation 
\begin{align}\label{d18.1}
Z_t =  B_t + \beta L_t^Z \, , \quad t \geq 0, 
\end{align}
has a unique strong solution (see \cite{HarrShepp,Lejay}). Here $L^Z$ is the symmetric local time of $Z$ at $0$. 
The process $Z$ is called skew Brownian motion.

Let $T = \inf\{t\geq 0: L_t^Z = 1\}$ and let $\{(B^k, T_k)\}_{k\in\Z}$ be an i.i.d.{} family with elements distributed as $(B,T)$. We now define an FBM $X$ as a strongly decomposable process based on 
$\{(B^k, T_k)\}_{k\in\Z}$, as in Definition \ref{d17.3}.

We can write as in \eqref{d18.1},
\begin{align*}
Z^k_t =  B^k_t + \beta L_t^k \, , \quad t \geq 0. 
\end{align*}
Let $L^X_t = L^k_t + k\beta$ for $t\in [S_k, S_{k+1}]$, $k\in \Z$. 
The analysis of the excursion process of $Z^k$ above 0 shows that 
the process $\{Y_t := X_{-t} + 2\beta L^X_{-t}, t\geq 0\}$ 
is Brownian motion.
The distribution of $L^Z_t$ is the same as that of $\max_{0\leq \leq t} B_t$ (see \cite{HarrShepp}) so $\E L^Z_t = \sqrt{2t/\pi} $. Hence, for $t\geq0$,
\begin{align}\label{d18.3}
\E X_{-t} = \E Y_t -\E(2\beta L^X_{-t}) = 2\beta \sqrt{2t/\pi}.
\end{align}
This shows that for different values of the parameter $\beta$, the distributions of FBM's $X$ are different. Moreover, \eqref{d18.1} and \eqref{d18.3} show that $X$ is two sided Brownian motion if and only if $\beta= 0$.  

For $\beta = 1$, $X_t$ is the same as in Example \ref{d17.1} 
because, in this case, $Z$ is reflected Brownian motion and $L^Z_t = \min_{0\leq s\leq t} B_s$ (see \cite{HarrShepp}).

If we let $S_u = \inf\{t: L^X_t = u\}$ for $u\in \R$ then for integer $u$, this definition of $S_u$ agrees with the definition of $S_k$ given in Definition \ref{d17.3}. It is easy to see that $S_u < S_v$ for $u<v$ and 
$\{X_{S_u+t}-X_{S_u},\,t \ge 0\}$ is Brownian motion
for every $u\in \R$. In other words, the process $X$ is Brownian motion as viewed from a family of random points $(S_u, X_{S_u})$ in space time; the cardinality of this family is the same as that of $\R$. We do not believe that such a family can be constructed for every FBM. For example, we doubt that it can be constructed for FBM presented in Section \ref{sec:FBMBBM}.
\end{example}

Since Example \ref{d17.1} is the ``extreme'' case of Example \ref{d17.2}, one 
may wonder whether properties of trajectories of $X_t$ in Example \ref{d17.1},
when $t\to -\infty$, display ``extreme'' possible behavior for trajectories
of any FBM. In other words, are path properties of 3-dimensional Bessel process
extreme among path properties of all FBM's? The answer is negative in every conceivable sense---see
Sections \ref{maxrange}-\ref{minrange}. 

\begin{example}\label{d18.11}
We will show that if $Z$ is a strongly decomposable FRW and $T$ is a random time such that $\{Z_{T+n},\,n\in \N_0\}$ 
is a simple symmetric random walk then this does not imply that the increments of 
$Z$ to the right and to the left of $T$ are independent. 
Let $Y$ be simple symmetric random walk starting from $Y_0=0$ and let
$S=\inf\{n \in \{2,3,...\}:
Y_n-Y_{n-1}=Y_{n-1}-Y_{n-2}\}$. Let $Z$ be strongly decomposable FRW constructed as a concatenation of independent copies of $(Y,S)$.  
It is easy to see that if $T \equiv -1$ then $\{Z_{T+n},\,n \in \N_0\}$ is a simple symmetric random walk and that the increments 
of $Z$ before and after time $T$ are not independent. 
Specifically, $Z_T-Z_{T-1}$ and $Z_T-Z_{T+1}$ are fully correlated.
We note parenthetically that $\{X_{-1-n}, \,n \in \N_0\}$ is 
not a simple symmetric random walk in this example.
If $X$ is the FBM constructed from $Z$ as in Remark \ref{d18.10} and $T\equiv -1$ then $\{X_{T+t} - X_T, t\geq 0\}$ and $\{X_{T-t} - X_T, t\geq 0\}$ are not independent because if the first process hits 1 before hitting $-1$ then the opposite is true of the second process.
This is a much weaker claim than that in Remark \ref{d18.6}.
\end{example}

\section{Maximum asymptotic range}\label{maxrange}

The main result of this section, Theorem \ref{d18.12}, states that the $\limsup$ of FBM in the backward direction can be arbitrarily 
large. By symmetry, the $\liminf$ can be arbitrarily small. Moreover,
both assertions can be true simultaneously. As a warm up,
we present two simple results that have short proofs.

\begin{proposition}\label{LIL}
Let $X$ be a decomposable FBM with associated sequence $S_k,\,k \in \Z$, in the notation of Definition \ref{d17.3}. Then we have
\begin{equation}\label{LILS}
\limsup_{n \to -\infty} \frac{X_{S_n}}{\sqrt{2 |S_n| \log \log |S_n|}}\le 1 \mbox{ a.s.}
\end{equation}
By symmetry, an analogous inequality holds for $\liminf$.
\end{proposition}

\begin{proof}
Define
$$
\ol X_{t}:= X_{S_n+S_{n-1}+t}-X_{S_n}-X_{S_{n-1}} \qquad \mbox{ if } S_{n-1} \le -t \le S_n,\, -n \in \N_0,
$$
and observe that $\{\ol X_t, \, t\ge 0\}$ is standard Brownian motion because, for each $-n\in \N_0$, we shifted the graph of $X$ between $S_{n-1}$ 
and $S_n$ by $(-S_{n-1}-S_n, -X_{S_{n}}-X_{S_{n-1}})$. Since $\ol X_{-S_n}=-X_{S_n}$ for all $-n \in \N_0$,
\eqref{LILS} follows from the usual law of the iterated logarithm for Brownian motion.
\end{proof}

\begin{proposition}\label{d22.5}
Let $X$ be FBM. Then we have
$$
\liminf_{t \to -\infty} \frac{X_t}{\sqrt{2|t|\log \log |t|}} \le 1 \mbox{ a.s.}
$$
\end{proposition}
\begin{proof}
Note the it will suffice to prove the claim for $X_t - X_0$ in place of $X_t$. Let $S_n$ be as in Definition \ref{d17.4}.
Fix an arbitrarily small $\ve >0$ and let 
$$
p_n :=\P(X_{S_n}-X_0\ge (1+\ve)\sqrt{2|S_n|\log \log |S_n|}).
$$
By the LIL and the fact that $S_n \to -\infty$ it follows that $p_n \to 0$. Passing to a subsequence, if necessary, for which the sum of $p_n$'s is finite, 
we see that, by the Borel-Cantelli Lemma, we have almost surely,
$$
X_{S_n}-X_0< (1+\ve)\sqrt{2|S_n|\log \log |S_n|}
$$
for infinitely many $n$, so the proposition follows.
\end{proof}

\begin{theorem}\label{d18.12}
For each increasing function $f:[0,\infty) \to [0,\infty)$ there exists a strongly decomposable FBM $X$ for which, a.s., 
\begin{align*}
\limsup_{t \to -\infty} \left( X_t - f(-t) \right) \ge 0 \qquad\mbox{ and }\qquad 
\liminf_{t \to -\infty} \left( X_t + f(-t) \right) \le 0. \end{align*}
\end{theorem}
 
\begin{proof}
We will assume without loss of generality that 
\begin{equation}\label{loglog}
f(x) \ge  2\sqrt{2x \log^+\log^+ x},\qquad x \ge 0,
\end{equation}
where $\log^+ x:=\max\{\log x,1\}$.

Our construction of $X$ will be based on a random variable $Y$ whose distribution will be specified in several steps. Suppose that $Y$ and a Brownian motion $B$
are defined on the same probability space and are independent. Let 
$$
T:=\inf\{t \ge 1: B_t-B_{t-1}=Y\}.
$$
Let $(Y_k,\,B^k,\, T_k)$, 
$k \in \Z$, be independent copies of $(Y,\,B,\,T)$ and define the $S_k$'s and $X$ as in Definition \ref{d17.3}.

We will later specify a sequence $\{n_k\}_{k\in \N_0}$ of non-negative real numbers strictly increasing to $\infty$.
We define the distribution of $Y$ by 
\begin{align*}
\P(Y=n_k)&=\P(Y=-n_k)=2^{-k^2-1}=:p_k,\qquad k \in \N,\\ 
\P(Y=0)&=1-2\sum_{k=1}^{\infty} p_k=:p_0.
\end{align*}
Let $K(m)$ be the largest negative integer $k$ for which $|Y_k|=n_m$ and define the events
$$
C_m:=\{K(m)>\max_{j>m}K(j)\}.
$$ 
Let $q_m = \sum_{j>m} p_j$.
It is elementary to see that $\P(C_m) = p_m/(p_m + q_m)$ so 
\begin{align}\label{d19.1}
\P(C_m^c) = \frac{q_m}{p_m+ q_m} \leq \frac{q_m}{p_m}
= \frac{\sum_{j>m} 2^{-j^2-1}}{2^{-m^2-1}}
\leq \frac{2\cdot 2^{-(m+1)^2-1}}{2^{-m^2-1}}
= 2^{-2m}.
\end{align}
Hence, $\sum_{m=1}^{\infty} \P(C_m^c) < \infty$ 
and, by the Borel-Cantelli Lemma, almost surely, all but finitely many of the $C_m$ occur. This means that there 
exists almost surely some random $m_0$ such that $K(m+1)< K(m)$ for all $m \ge m_0$.   

We will show that, for suitably chosen $\{n_k\}_{k\in \N_0}$, each of the inequalities
\begin{equation}\label{ypsilon}
Y_{-k} \ge 2f(-S_{-k+1}+1)\qquad \mbox{ and }
\qquad Y_{-k} \le -2f(-S_{-k+1}+1)
\end{equation}
holds for infinitely many $k \in \N$ almost surely. Once we have shown this, then the theorem follows from Proposition \ref{LIL}
and \eqref{loglog}. 
By symmetry, it suffices to show the first of the two inequalities in \eqref{ypsilon}. 

For a given function $f$, we will define the numbers $n_k$ inductively, starting with $n_0=0$.   
Note that the law of $S_{K(m)+1}$ conditioned on $C_m$ does not depend on the choice of $\{n_k\}_{k\ge m}$. 
For $m \in \N$, let $n_m$ be so large that $n_m>n_{m-1}$ and 
$$
\P( n_m \ge 2 f(-S_{K(m)+1}+1)\mid C_m) \ge 1-2^{-m}. 
$$
Define
$$
A_m:=\{n_m \ge 2 f(-S_{K(m)+1}+1)\}.
$$
Then,
$$
\P(A_m) \ge \P(A_m\mid C_m)\P(C_m) \ge (1-2^{-m}) \P(C_m)
$$
and, therefore, in view of \eqref{d19.1},
$$
\sum_{m=1}^{\infty} \P(A_m^c) \le \sum_{m=1}^{\infty} (1-\P(C_m)+2^{-m}\P(C_m))
=\sum_{m=1}^{\infty} (\P(C_m^c)+2^{-m}\P(C_m))<\infty, 
$$
which implies, by the Borel-Cantelli Lemma, that all but finitely many of the events $A_m$
occur. Next, let
$$
V_m:=\{Y_{K(m)}\ge 0\} =\{Y_{K(m)}=n_m\}.
$$
Then
\begin{equation}\label{Am}
A_m \cap V_m \subseteq \{Y_{K(m)}\ge  2 f(-S_{K(m)+1}+1)\}. 
\end{equation}
Since the $V_m$'s are i.i.d. and $\P(V_m)=1/2$, almost surely infinitely many of the $V_m$'s occur and, therefore, infinitely many of the
$A_m \cap V_m$'s occur. 
Together with \eqref{Am} this implies \eqref{ypsilon} and the theorem is proved. 
\end{proof}

\section{Minimum asymptotic range}\label{minrange}

In the previous section, we showed that the $\limsup$ of an FBM, as $t\to-\infty$, can be ``arbitrarily large.'' In this section we will show 
that the $\liminf$ of an FBM in the backward direction cannot be arbitrarily 
large. We will consider regions in space-time of the form $\cR:=\{(t,x):t<0, c_1 \sqrt{|t|} < x < c_2 \sqrt{|t|}\}$ into which paths of an FBM may fit, at least asymptotically.  
Roughly speaking, there exist FBM's whose paths stay inside $\cR$ as $t\to -\infty$ if
and only if $c_1$ and $c_2$ are not too close to each other. Here, being ``close'' is a condition more complicated than a bound on $c_2-c_1$. Examples of ``critical pairs'' of $(c_1,c_2)$ are $(1,\infty)$, $(-1,1)$ and $(0, 2.12)$ (the last number is approximate).
On the technical side, this section is closely related to the problem of slow points for Brownian motion studied in \cite{Davis83,GP83,Perkins}. We will mostly cite \cite{Perkins}.

\begin{remark}\label{d23.2}
In this remark, we collect some results from \cite[p.~371]{Perkins}.
Let
\begin{align*}
\cC &= \{(c_1, c_2) : - \infty \leq c_1 < c_2 \leq \infty\},\\
\cA &= \frac12 \left( \frac {d^2}{dx^2} - x \frac d {dx}\right),
\end{align*}
and $m(dx) = 2 e^{-x^2/2}dx$. For each $(c_1,c_2)\in\cC$ , there is a complete orthonormal
system in $L^2([c_1,c_2],m)$ of eigenfunctions of the Sturm-Liouville problem
\begin{align*}
\cA \psi = \lambda \psi, \qquad \psi(c_i)=0, i=1,2,
\qquad \text{  if  } |c_i | < \infty,
\end{align*}
whose corresponding eigenvalues are simple and non-positive. 
Let $- \lambda_0(c_1,c_2)$ denote the largest eigenvalue.
The corresponding eigenfunction $\psi (c_1, c_2, x)$ can be assumed to be strictly positive on
$(c_1, c_2)$. 

The function $\lambda_0$ is continuous on $\cC$ and strictly positive on $\cC\setminus \{-\infty,\infty\}$. The function
$\lambda_0(\,\cdot\,,c_2)$ is strictly increasing on $[-\infty,c_2)$ and $\lambda_0(c_1, \,\cdot\,)$ is strictly decreasing on $(c_1, \infty]$.
\end{remark}

\begin{remark}\label{d27.1}
For our results, just like for many results in \cite{Davis83,GP83,Perkins}, the critical value of $\lambda_0$ is 1, so it is of interest to know for which values of $c_1$ and $c_2$ we have  $\lambda_0(c_1,\,c_2)=1$. Some examples of such pairs are
$(-\infty,-1),\,(-1,1)$ and $(1,\infty)$ 
(see \cite[Prop.~1]{Perkins}). It is natural to ask what $c'_2$ satisfies $\lambda_0(0,\,c'_2)=1$.
The approximate value of such $c'_2$ is $2.12411$.
We found this value as follows.
Observe that 
$$
\psi(x) = 2 \exp(x^2/2) x + \sqrt{2 \pi} \erfi(x/\sqrt{2})
 - \sqrt{2 \pi} x^2 \erfi(x/\sqrt{2})
$$
satisfies the equation $(1/2) (\psi''(x) - x \psi'(x)) = -\psi(x)$
and $\psi(0) = 0$. Here 
$\erfi(x) = -i \erf(i x)$ and
$\erf(x) = (2/\sqrt{\pi})\int_0^x e^{- t^2} dt$.
The function $\psi$ is strictly positive on an interval $(0,\,c'_2)$ and vanishes at the endpoints of this interval. We determined that $c'_2 \approx 2.12411$ 
by solving $\psi(x) = 0$ numerically.

We note that $c'_2$ appears to be the same as $c(3) $ on page 376 in \cite{Perkins}.
We offer an informal explanation for the coincidence. The constant $c(3) $ corresponds
to 3-dimensional Bessel process staying under a parabola. This problem
can be equivalently represented as that about 
1-dimensional Brownian motion staying between 0 and the same parabola,
because 1-dimensional Brownian motion conditioned not to hit 0 is 
3-dimensional Bessel process.
\end{remark}

\begin{theorem}\label{d30.5}
{\rm{(i)}} If $\lambda_0(c_1,c_2)\ge 1$ and $X$ is FBM, then 
$$
\P\left(\{\limsup_{t \to -\infty} {X_t}/{\sqrt{|t|}}\ge c_2\} \cup 
\{\liminf_{t \to -\infty} {X_t}/{\sqrt{|t|}}\le c_1\}\right)=1.
$$

{\rm{(ii)}} If $\lambda_0(c_1,c_2)\leq 1$, then there exists a decomposable FBM $X$ such that 
\begin{align}\label{d26.4ii}
\P\left(\{\limsup_{t \to -\infty} {X_t}/{\sqrt{|t|}}\le c_2\} \cap 
\{\liminf_{t \to -\infty} {X_t}/{\sqrt{|t|}}\ge c_1\}\right)=1.
\end{align}

{\rm{(iii)}} If $c_1 \leq 0 \leq c_2$ and $\lambda_0(c_1,c_2)< 1$, then there exists a decomposable FBM $X$ such that 
\begin{align}\label{d26.4iii}
\P\left(\{ c_1{\sqrt{|t|}} \le X_t\le c_2{\sqrt{|t|}}\; \forall t \le 0\}\right)=1.
\end{align}
\end{theorem}

\begin{proof}
(i) Fix $-\infty<c_1<c_2<\infty$ such that $\lambda_0(c_1,c_2) >1$ and let $B$ denote standard Brownian motion. For $a \ge 0$ and $n \in \N_0$, let 
\begin{align*}
F(n,a) & := \{B(s) \in [c_1 \sqrt{s}-a,c_2 \sqrt{s}+a]\,\forall \, 0 \le s \le n\}, \\
r(n,a,c_1,c_2)&:= \P(F(n,a)). 
\end{align*}
We will show that if  $\lambda_0(c_1,c_2)>1$, then for any $a>0$, 
\begin{align}\label{d23.1}
\sum_{n=1}^{\infty}r(n,a,c_1,c_2) < \infty.
\end{align}

By the 
continuity of $\lambda_0(\,\cdot\,,\,\cdot\,)$ 
(see Remark \ref{d23.2}), we can choose $\delta>0$ such that $\lambda_0(c_1-\delta,c_2+\delta)>1$.
Let $A>0$ be such that $a+c_2\sqrt{s} \le (c_2+\delta)\sqrt{s}$ and $-a+c_1\sqrt{s} \ge (c_1-\delta)\sqrt{s}$ hold for all $s \ge A$. Then, using Brownian scaling, we obtain
\begin{align*}
r(n,a,c_1,c_2)&\le \P\{B(s) \in [(c_1-\delta)\sqrt{s},(c_2+\delta)\sqrt{s}]\, \forall A \le s \le n\}\\
&=\P\{B(u) \in [(c_1-\delta)\sqrt{u},(c_2+\delta)\sqrt{u}]\, \forall 1 \le u \le n/A\}\\
&\sim K_1(c_1-\delta,c_2+\delta)  \Big( \frac nA \Big)^{-\lambda_0(c_1-\delta,c_2+\delta)},
\end{align*} 
where $K_1(c_1-\delta,c_2+\delta) \in (0,\infty)$ and the asymptotic equivalence follows from \cite[Lem.~10(b)]{Perkins}. Since the right hand side is summable the proof of 
\eqref{d23.1} is complete.

For $a,b \ge 0$, let
$$
T_{a,b,c_1,c_2}:=\inf\{ t \ge a \vee b: \, B(s) \in [B(t)+c_1\sqrt{t-s},\, B(t) +c_2 \sqrt{t-s}]\,\forall  \, 0 \le s \le t-b\}.
$$
We will show that,
\begin{align}\label{d24.1}
\lim_{b \to 0} \P(T_{2,b,c_1,c_2} < \infty)=0.
\end{align}
For $0 \le b_1 < b_2$, the (random) sets
$$
\Lambda(b_1,b_2):=\{t \in [0,1]: B(s) \in [B(t)+c_1 \sqrt{s-t},B(t)+c_2 \sqrt{s-t}]\; \forall s \in [t+b_1,t+b_2]\}
$$
are compact and for every $b_2>0$, 
$\bigcap_{b_1\in(0,b_2)} \Lambda(b_1,b_2) = \emptyset$, a.s., by \cite[Thm.~2(a)]{Perkins}.
Therefore, there exists a random $b_0=b_0(b_2)>0$ such that $\Lambda(b_1,b_2)=\emptyset$ for all 
$0 \le b_1 <b_0$.
Hence, if we write 
$q(b_1,b_2,c_1,c_2) = \P(\Lambda(b_1,b_2) \ne \emptyset)$ then
\begin{equation}\label{q}
\lim_{b_1 \to 0} q(b_1,b_2,c_1,c_2) =\lim_{b_1 \to 0} \P(\Lambda(b_1,b_2)\neq \emptyset)=0.  
\end{equation}

For each $n \in \N,\, n\ge 2$, the processes $\{B'_t := B_{n+1 - t} - B_{n+1}, t\in[0,2]\}$ and $\{B''_t := B_{n-1 - t} - B_{n-1}, t\in[0,n-1]\}$ are independent Brownian motions. Note that for $b \in [0,1)$ we have
\begin{align*}
\{T_{2,b,c_1,c_2} \in [n,n+1)\}
\subset \{\Lambda(b,1, B')\neq \emptyset\} \cap
F(n-1,(|c_1|+|c_2|)\sqrt{2}, B''),
\end{align*}
where $\Lambda(b,1, B') $ and $F(n-1,(|c_1|+|c_2|)\sqrt{2}, B'')$ denote $\Lambda(b,1) $ and $F(n-1,(|c_1|+|c_2|)\sqrt{2})$
defined relative to the processes $B'$ and $B''$, resp., in place of $B$.
We obtain, 
\begin{align*}
\P(T_{2,b,c_1,c_2} < \infty)
&=\sum_{n=2}^{\infty}\P(T_{2,b,c_1,c_2} \in [n,n+1))\\
&\le \sum_{n=2}^{\infty} 
r(n-1,(|c_1|+|c_2|)\sqrt{2},c_1,c_2)\,q(b,1,c_1,c_2)\\
&= q(b,1,c_1,c_2)\, \sum_{n=2}^{\infty} r(n-1,(|c_1|+|c_2|)\sqrt{2},c_1,c_2).  
\end{align*}
The last sum is finite (and independent of $b$) by \eqref{d23.1}. We conclude that \eqref{d24.1} holds in view of \eqref{q}.

Assume that
\begin{align}\label{d24.2}
\P\left(\{\limsup_{t \to -\infty} {X_t}/{\sqrt{|t|}}\le c_2\} \cap 
\{\liminf_{t \to -\infty} {X_t}/{\sqrt{|t|}}\ge c_1\}\right)=:q>0.
\end{align}
To prove part (i) of the theorem, it will suffice to show that this assumption leads to a contradiction.

Recall that we have
chosen $\delta>0$ such that $\lambda_0(c_1-\delta,c_2+\delta)>1$. Assuming \eqref{d24.2}, we can 
find some $M \in (-\infty, 0)$ such that 
\begin{align}\label{d24.5}
\P\left(\{\sup_{t \le M} {X_t}/{\sqrt{|t|}}\le c_2+\delta\} \cap 
\{\inf_{t \le M} {X_t}/{\sqrt{|t|}} \ge c_1-\delta\}\right)\ge \frac q2.
\end{align}
Consider $\wt M<M$, whose value will be specified later. 
Since $X$ is FBM, there exists a random time $S$ such that 
$\P(S \le \wt M) \ge 1-\frac q4$ and $\{X_{S+t}-X_S,\,t \ge 0\}$ is Brownian motion. 
Then, 
\begin{align}\label{d24.3}
\P&\left(\{\sup_{t \le M} {X_t}/{\sqrt{|t|}}\le c_2+\delta\} \cap 
\{\inf_{t \le M} {X_t}/{\sqrt{|t|}} \ge c_1-\delta\}\right)\\
&\le \P(S > \wt M) + \P(T_{-\wt M,-M,c_1-\delta,c_2+\delta}<\infty)\nonumber \\
&\le \P(T_{-\wt M,-M,c_1-\delta,c_2+\delta}<\infty) +\frac q4 \nonumber \\
&=\P(T_{-\alpha\wt M,-\alpha M,c_1-\delta,c_2+\delta}<\infty) +\frac q4, \nonumber
\end{align}
for any $\alpha>0$, where the last equality follows from Brownian scaling. 
By \eqref{d24.1}, we can make 
$\alpha >0$ so small that
$\P(T_{2,-\alpha M,c_1-\delta,c_2+\delta}<\infty) < q/8$.
Then we 
choose $\wt M$ so that $-\alpha \wt M = 2$. The left hand side of \eqref{d24.3} is therefore less than $3q/8$, which contradicts \eqref{d24.5}. 
This proves part (i) in case $\lambda_0(c_1,c_2)>1$ and $-\infty<c_1<c_2<\infty$.  

If $\lambda_0(c_1,c_2)=1$ and $-\infty<c_1<c_2<\infty$, then $\lambda_0(c_1+\varepsilon, c_2-\varepsilon)>1$
for every $\varepsilon \in (0, \frac 12 (c_2-c_1))$,
by Remark \ref{d23.2}. We have already shown that, for every $\eps>0$,
$$
\P\left(\{\limsup_{t \to -\infty} {X_t}/{\sqrt{|t|}}\ge c_2-\eps\} \cup 
\{\liminf_{t \to -\infty} {X_t}/{\sqrt{|t|}}\le c_1+\eps\}\right)=1.
$$
This implies that 
$$
\P\left(\{\limsup_{t \to -\infty} {X_t}/{\sqrt{|t|}}\ge c_2\} \cup 
\{\liminf_{t \to -\infty} {X_t}/{\sqrt{|t|}}\le c_1\}\right)=1,
$$
and completes the proof of part (i) in case $-\infty<c_1<c_2<\infty$. The case $c_2=\infty$ is 
treated by applying the previous result to a sequence $c_{2,n} \to \infty$ (and similarly for $c_1=-\infty$).

\bigskip
(ii), (iii) 
According to \cite[Thm.~2(a)]{Perkins},
\begin{align}\label{d26.5}
\P(\exists t \ge 0,\,\Delta>0 : B(t+h)-B(t) \in [c_1 \sqrt{h},\,c_2 \sqrt{h}]\  \forall h \in [0,\Delta])=
\left\{
\begin{array}{ll} 
0&{\rm{ if }}\,\lambda_0(c_1,c_2)>1,\\
1&{\rm{ if }}\,\lambda_0(c_1,c_2)<1.
\end{array}
\right.
\end{align}

First suppose that $c_1 \leq 0< c_2$ and $\lambda_0(c_1,c_2)<1$. By Remark \ref{d23.2} it suffices to prove  \eqref{d26.4ii} and  \eqref{d26.4iii} in case $c_1>-\infty$ and $c_2<\infty$.
These assumptions, \eqref{d26.5}, invariance of Brownian motion under time reversal, support theorem, and standard 
arguments imply that
\begin{align*}
\P(\exists t \in[1,2] : B(t-s)-B(t) \in [c_1 \sqrt{s},\,c_2 \sqrt{s}]\  \forall s \in [0,t])
>0.
\end{align*}
Another easy application of the support theorem and Brownian scaling allows to strengthen the above claim to the following.
If $\lambda_0(c_1,c_2)<1$ and $\delta \in (0,c_2)$ then there exists $p_1>0$ such that for every $a\in(0,\infty)$,
\begin{align}\label{d25.1}
&\P\big(\exists t \in[a/2,a] : B(t)-B(0) \in  [-\frac{\delta}2 \sqrt{t},-\frac{\delta}4 \sqrt{t}] \\ 
&\qquad \text{  and  }
B(t-s)-B(t) \in [c_1 \sqrt{s},\,c_2 \sqrt{s}]\  \forall s \in [0,t]\big)\nonumber\\
&= \P\big(\exists t \in[1,2] : B(t)-B(0) \in   [-\frac{\delta}2 \sqrt{t},-\frac{\delta}4 \sqrt{t}]  \nonumber \\
&\qquad \text{  and  }
B(t-s)-B(t) \in [c_1 \sqrt{s},\,c_2 \sqrt{s}]\  \forall s \in [0,t]\big) \nonumber\\
&=p_1>0. \nonumber
\end{align}
Let $u_n = \exp(\exp(\exp(n)))$ for $n\in \N$. Note that for large $n$ (depending on $\delta$),
\begin{align}\label{d25.3}
(\delta/4)\sqrt{(u_{n+1} - u_n)/2}
\geq (\delta/8)\sqrt{u_{n+1}}
\geq u_n > 3\sqrt{2 u_n \log\log u_n}.
\end{align}
The processes $\{Y^n(t):= B(u_n+t) - B(u_n), t\in [0, u_{n+1} - u_n]\}  $ are independent Brownian motions. 
The events
\begin{align*}
F_n :=
\big\{&\exists t \in[(u_{n+1} - u_n)/2,u_{n+1} - u_n] : Y^n(t)-Y^n(0)\in[-\frac{\delta}2 \sqrt{t},-\frac{\delta}4 \sqrt{t}]    \\
& \text{  and  }
Y^n(t-s)-Y^n(t) \in [c_1 \sqrt{s},\,c_2 \sqrt{s}]\  \forall s \in [0,t]\big\}
\end{align*}
are independent and each one of them has probability $p_1$, by \eqref{d25.1}. Hence, infinitely many events $F_n$ occur, a.s. 
By the law of the iterated logarithm, a.s., for all sufficiently large $n$,
\begin{align}\label{LiL}
\sup_{0\leq s \leq u_n} |B(s)| <
2\sqrt{2 u_n \log\log u_n}.
\end{align}
If \eqref{d25.3}, \eqref{LiL} and $F_n$ hold 
then the following event occurs,
\begin{align*}
\big\{&\exists t \in[(u_{n+1} + u_n)/2,u_{n+1}] : B(t)-B(0) \in [-\delta \sqrt{t},0] \\
& \text{  and  }
B(t-s)-B(t) \in [c_1 \sqrt{s},\,c_2 \sqrt{s}]\  \forall s \in [0,t]\big\}.
\end{align*} 
Since infinitely many events $F_n$ occur, a.s., we conclude that
if $\lambda_0(c_1,c_2)<1$, $\delta >0$ and $a<\infty$ then,
\begin{align}\label{d25.2}
\P\big(&\exists t \geq a : B(t)-B(0) \in [-\delta \sqrt{t},0]  \\
&\text{  and  }
B(t-s)-B(t) \in [c_1 \sqrt{s},\,c_2 \sqrt{s}]\  \forall s \in [0,t]\big)
=1. \nonumber
\end{align}

We recall the definition of decomposable FBM $X$ from Definition \ref{d17.3}.
Given $(B^k,T_k)$, $k\in \Z$,
let $S_0=0$, and use the conditions $S_{k+1} - S_{k} = T_k$ to define $S_k$ for $k\in \Z$.
Let $X$ be the unique continuous process such that $X_0 =0$ and
$X_{S_k +t} - X_{S_k} = B^{k}_{t}$ for $t\in [0, T_{k})$, $k \in\Z$.

Suppose that
$\{B^k_t, t\geq 0\}$, $k\in \Z  $, are independent Brownian motions.
Let $T_k \equiv 1$ for $k\in \N_0$. To define $T_k$ for negative $k$, observe that
for given $-\infty<c_1<c_2<\infty$ satisfying 
$c_1 \leq 0 < c_2$ and
$\lambda_0(c_1,c_2)<1$ we can find some $\varepsilon>0$ such that $c_2-\varepsilon>0$ 
and $\lambda_0(c_1,c_2-\varepsilon)<1$ by Remark \ref{d23.2}. For $-k \in \N$, let
$$
c_2(k):=c_2-\varepsilon+\frac {\varepsilon}{|k|},\,\delta(k):=\varepsilon\Big(\frac 1{|k|}-\frac 1{|k-1|}\Big).
$$ 
For $-k \in \N$, we define
\begin{align}
T_{k}: = 
\inf\{&t \geq 1:
B^{k}(t)-B^{k}(0) \in [-\delta(k) \sqrt{t},0] \nonumber \\
&\text{  and  }
B^{k}(t-s)-B^{k}(t) \in [c_1 \sqrt{s},\,c_2(k) \sqrt{s}]\  \forall s \in [0,t] \},\label{d30.1}
\end{align}
and note that $T_k < \infty$ a.s., by \eqref{d25.2}.

By construction, we have for $-k \in \N_0$ and $s \in [S_{k-1},S_k]$:
\begin{align*}
X_s &\le \sum_{i=1}^{-k} \delta (-i)\sqrt{S_{-i+1}-S_{-i}}+c_2(k-1)\sqrt{S_{k}-s}\\
&\le\left(  \sum_{i=1}^{-k} \delta (-i)+c_2(k-1)\right)\sqrt{-s}=c_2 \sqrt{-s},
\end{align*}
and $X_s \geq c_1\sqrt{-s}$ for all $s \leq 0$, so \eqref{d26.4iii} and hence \eqref{d26.4ii} follow in case $c_1 \leq 0<c_2$ and $\lambda_0(c_1,c_2)<1$. 

Now assume that $\lambda_0(c_1,c_2)=1$ (and still $c_1\leq 0<c_2$). Then, by Remark  \ref{d23.2},  $\lambda_0(c_1,c_2+\varepsilon)<1$ for every $\varepsilon >0$. 
Consider the FBM $X$ constructed in the previous paragraph but with $(c_1,c_2)$ replaced by $(c_1,c_2+1)$ and $\varepsilon=1$. 
Let $c_2(k):=c_2 +\frac 1{|k|}$ and $\delta(k):=\frac 1{|k|}-\frac 1{|k-1|}$. Then, we have for fixed $-m \in \N_0$ and $k \le m$ and
$s \in [S_{k-1},S_k]$:
\begin{align*}
X_s &\le \sum_{i=-m+1}^{-k} \delta (-i)\sqrt{S_{-i+1}-S_{-i}}+c_2(k-1)\sqrt{S_{k}-s}\\
&\le\left(  \sum_{i=-m+1}^{-k} \delta (-i)+c_2(k-1)\right)\sqrt{S_{m}-s}=c_2(m-1) \sqrt{S_{m}-s},
\end{align*}
and hence
$$
\limsup_{s \to -\infty} \frac{X_s}{\sqrt{|s|}}\le c_2(m-1),
$$
for every $-m \in \N_0$. Since $c_2(m)$ converges to $c_2$ as $m \to -\infty$ and since $X_s\geq c_1\sqrt{-s}$ for all $s \leq 0$,
the proof of (ii) and (iii) is complete in the case $c_1\leq 0<c_2$.

\bigskip
Next we consider the case when $\lambda_0(c_1,c_2)<1$ but it is not true that 
$c_1 \leq 0 < c_2$. By symmetry, we may and will assume that $0 < c_1 \leq c_2 \leq \infty$ and (again by Remark \ref{d27.1}) we can and will assume that 
$c_2<\infty$.  
In this case the reasoning in the previous case will not work because the region $\{(s,x): c_1 \sqrt{s} < x <c_2 \sqrt{s},\,s \ge 0\}$ is not convex. 
Consequently, our argument is more complicated in the present case. 

By Remarks \ref{d23.2} and \ref{d27.1}, $c_1 < 1$ and $c_2>2$.

Suppose that 
$\delta,\phi,a>0$. Let $c_1''$ be any number such that $c_1 < c_1''< 1$, $c_1''-c_1<\delta/4$ and $\lambda_0(c_1'',c_2)<1$.  
We will prove that for every $b\in(a,\infty)$,
\begin{align}\label{d27.5}
\P\big(&\exists t \geq b : B(0)-B(t)\in [c_1\sqrt{t},\,  (c_1+\delta) \sqrt{t}]\\
&\text{  and  }
B(t-s)-B(t) \in [c_1 \sqrt{s} ,\,c_2 \sqrt{s}]\  \forall s \in [a,\,t] \nonumber\\
&\text{  and  }
\,B(t-s)-B(t) \in [c_1'' \sqrt{s},\,c_2 \sqrt{s}\land(c_1'' \sqrt{s} + \phi)]\  \forall s \in [0,a]\}=1. \nonumber
\end{align}
In the proof we will need several strictly positive constants, namely $c_1',c_2',\varepsilon$ and $\hat \varepsilon$. We suppose that they satisfy the following constraints:
\begin{align}
&c_1''<c_1'< 1 < 2 <c_2'<c_2,\;c_1'-c_1''<\delta/4 \mbox{ and } \lambda_0(c_1',c_2')<1,\label{first}\\
&2\varepsilon+\phi+c_1''\sqrt{a+\hat \varepsilon} -c_1'\sqrt{a} \le 0,\label{third}\\
&(c_1''-c_1)\sqrt{a}\ge 2 \varepsilon,\label{fourth}\\
&(c_1''-c_1')\sqrt{a} \ge \varepsilon + c_1\sqrt{2a+\hat \varepsilon}-c_1'\sqrt{2a}.\label{fifth}
\end{align} 
It is easy to see that these constraints can all be fulfilled provided that $\phi$ is sufficiently small which we can and will assume without loss of generality.
One can first choose $c_1'$ and $c_2'$ satisfying \eqref{first}. Then one can choose $\phi,\varepsilon,\hat \varepsilon>0$ such that 
conditions \eqref{third}-\eqref{fifth} hold.

Let $\kappa:=\delta/2$. The following claim can be proved in the same way as
\eqref{d25.1}.
There exists $p_1>0$ such that for every $u\in(0,\infty)$,
\begin{align}\label{d30.10}
\P\big(&\exists t \in [u/2,u] :
B(t-s)-B(t) \in [c_1' \sqrt{s} ,\,h_u(s) ]\  \forall s \in [0,t]
\big)\\
&= \P\big(\exists t \in[1,2] : B(t-s)-B(t) \in [c_1' \sqrt{s},\,(c_2' \sqrt{s}) \wedge (c_1'\sqrt{s}+\kappa)]\  \forall s \in [0,t]\big) = p_1, \nonumber
\end{align}
where $h_u(s):=(c_2' \sqrt{s})\wedge (c_1'\sqrt{s}+\kappa \sqrt{u/2})$.   
Let $u>4a$ and define
\begin{align*}
U = \inf\{& \theta \in [\frac u2 -2a,u-2a] : \exists x \in \R \\
&\text{  such that  } 
B(\theta+2a-s)-x \in [c_1' \sqrt{s} ,\,h_u(s) ]\  \forall s \in [2a,\theta+2a]\}.
\end{align*}

Note that $U$ is a stopping time for $B$ (with the convention that $\inf \emptyset = \infty$) and that
\begin{align}\label{d27.7}
\P(U< \infty) \geq  p_1
\end{align}
for all $u>4a$ by \eqref{d30.10}.

On the set $\{U<\infty\}$ define $X^*$ as the largest number $x$ such that $B(U+2a-s)-x \in [c_1' \sqrt{s} ,\,h_u(s) ]$ for all $s \in [2a,U+2a]$ 
and let $X^*:=\infty$ on $\{U=\infty\}$. Observe that $X^*$ is $\F_U$-measurable.

On the set $\{U<\infty\}$ let $V(t):= B(U+t)-B(U)$ for $t\geq 0$. Since $B(U)-X^*$ is bounded from above and below by a deterministic constant (which does not depend on $u$) there 
exists $p_2>0$ (not depending on $u$) such that for all $ u > 4a$, on $\{U<\infty\}$,
\begin{equation}\label{erstens}
\P \left( \left|V(s)- (B(U)-X^*-c_1'\sqrt{a}) \frac{\sqrt{2a-s}-\sqrt{2a}}{\sqrt{a}(\sqrt{2}-1)} \right|\le \varepsilon \, \forall s \in [0,a] \mid \F_U \right) \ge p_2. 
\end{equation}
Further, there exists some $p_3>0$ (which does not depend on $u$) such that  for all $u>4a$ we have
\begin{align}\label{zweitens}
\P\big( \exists &\tau \in [2a,2a+\hat \varepsilon]: V(\tau -s)-V(\tau)  \in [c_1'' \sqrt{s},\,c_2 \sqrt{s}\land(c_1'' \sqrt{s}+\phi)]\  \forall s \in [0,\tau-a]     
\mid \F_{U+a} \big)\\ &\ge p_3  \quad\mbox{ on } \{U<\infty\}.\nonumber        
\end{align}
Let $G_1$ be the intersection of the set $\{U<\infty\}$ and the two sets inside the conditional probabilities in \eqref{erstens} and \eqref{zweitens}. 
By the strong Markov property we have $\P(G_1)\ge p_1 p_2 p_3$ for all $u > 4a$. Define
\begin{align*}
G_2 &:=\{\exists t \in [\frac u2,\,u+\hat \varepsilon]: B(0)-B(t) \in [c_1\sqrt{t},\,  (c_1+\delta) \sqrt{t}]\\
&\hspace{1cm}\text{  and  }
B(t-s)-B(t) \in [c_1 \sqrt{s} ,\,c_2 \sqrt{s}]\  \forall s \in [a,\,t]\\
&\hspace{1cm}\text{  and  }
\,B(t-s)-B(t) \in [c_1'' \sqrt{s},\,c_2 \sqrt{s}\land(c_1'' \sqrt{s} + \phi)]\  \forall s \in [0,a]\}.
\end{align*}
Once we know that (for a given $u>4a$) we have 
$G_1 \subseteq G_2$ then we obtain $\P(G_2)\ge p_1 p_2 p_3$. To see that $G_1 \subseteq G_2$ let  
$\tau$ be as in \eqref{zweitens} and $t:=U+\tau$.  Then the last property of $G_2$ clearly holds and the second one holds at least for 
$s \in [a,\,\tau -a]$. Now let $s\in [\tau-a,\tau]$. Then
\begin{align*}
B(t-s)-B(t)&=B(t-s) -B(t-\tau + a)+B(t-\tau + a)-B(t) \\
&= [V(\tau-s) - V(a)] +[B(t-\tau + a)-B(t)]\\
&\le \left[ 2 \eps + (B(U)-X^*-c_1'\sqrt{a}) \frac{\sqrt{s-\tau+2a}-\sqrt{a}}{\sqrt{a}(\sqrt{2}-1)}\right]
+ [c_1''\sqrt{\tau - a}  + \phi] \\
&\le 2\varepsilon + \frac{c_2'\sqrt{2}-c_1'}{\sqrt{2}-1}
( \sqrt{s-\tau+2a}-\sqrt{a})+ c_1''\sqrt{\tau - a}  + \phi\\
& \le c_2'\sqrt{s-\tau+2a}\\
& \le c_2\sqrt{s}.
\end{align*}
The second to last inequality holds for $s=\tau$ by \eqref{third}. Since the derivative with respect to $s$ of the left hand side is greater than that of the right hand side, the inequality holds for all $s\in [\tau-a,\tau]$. Further,
\begin{align*}
B(t-s)-B(t)&=B(t-s) -B(t-\tau + a)+B(t-\tau + a)-B(t) \\
&= [V(\tau-s) - V(a)] +[B(t-\tau + a)-B(t)]\\
&\ge \left[ -2 \eps + (B(U)-X^*-c_1'\sqrt{a}) \frac{\sqrt{s-\tau+2a}-\sqrt{a}}{\sqrt{a}(\sqrt{2}-1)}\right]
+ [c_1''\sqrt{\tau - a}  ] \\
&\ge - 2\varepsilon + c_1'(\sqrt{s-\tau + 2a}-\sqrt{a})+c_1''\sqrt{\tau - a} \ge c_1\sqrt{s}.
\end{align*}
The last inequality holds for $s=\tau - a$ by \eqref{fourth}. Since the derivative of the left hand side with respect to $s$ is greater than that of the right hand side for $s \ge \tau - a$, the last inequality holds for all $s\in [\tau-a,\tau]$. Next, we consider the case $s \in [\tau,t]$. We claim that
\begin{align*}
B(t-s)-B(t)&\ge X^*+c_1'\sqrt{2a-\tau+s}-B(t)\\
&\ge B(U+a)-c_1'\sqrt{a}-\varepsilon + c_1'\sqrt{2a-\tau+s}-B(t)\\
&\ge c_1''\sqrt{\tau -a}-c_1'\sqrt{a}-\varepsilon+  c_1'\sqrt{2a-\tau+s}\\
&\ge c_1\sqrt{s}.
\end{align*}
The first inequality follows from the definition of $X^*$. The second inequality follows from 
the condition in \eqref{erstens} applied with $s=a$. The third inequality follows from \eqref{zweitens} applied with $s=\tau - a$. The last inequality holds for $s=\tau$ by \eqref{fifth}. It holds for $s\ge \tau$ because the derivative of the left hand side is greater than that of the right hand side for $s\ge \tau$.
The following inequalities hold for similar reasons,
\begin{align*}
B(t-s)-B(t)&\le X^*+h_u(2a-\tau+s)-B(t)\\
&\le B(U+a)-c_1'\sqrt{a}+\varepsilon +h_u(2a-\tau+s) -B(t)\\
&\le c_1''\sqrt{\tau -a}+\phi-c_1'\sqrt{a}+\varepsilon+ h_u(2a-\tau+s)\\   
&\le c_1''\sqrt{a+\hat \varepsilon}+\phi-c_1'\sqrt{a}+\varepsilon+c_2'\sqrt{s}\\
&\le c_2\sqrt{s}.   
\end{align*}
The last inequality holds by \eqref{third} since $s \ge \tau \ge 2a$. Finally, for $s=t$, we obtain in the same way
\begin{align*}
B(t-s)-B(t)&\le  c_1''\sqrt{\tau -a}+\phi-c_1'\sqrt{a}+\varepsilon+ h_u(2a-\tau+t)\\ 
&\le  c_1''\sqrt{\tau -a}+\phi-c_1'\sqrt{a}+\varepsilon+  c_1'\sqrt{2a-\tau+t}+\kappa\sqrt{u/2}\\
&\le  c_1''\sqrt{a+\hat \varepsilon}+\phi-c_1'\sqrt{a}+\varepsilon+  c_1'\sqrt{t}+\kappa\sqrt{t}\\
&\le(c_1+\delta)\sqrt{t}.
\end{align*}
The last inequality can be derived from
\eqref{third} and the following facts: $c'_1 - c_1 \leq \delta/2$ and $\kappa = \delta/2$.

We have verified that all conditions in the definition of $G_2$ hold.
Therefore
$G_1 \subseteq G_2$ and the proof that $\P(G_2)\ge p_1 p_2 p_3$ (for all $u > 4a$) is complete.

The rest of the proof of \eqref{d27.5} is analogous to the argument showing that \eqref{d25.1} 
implies \eqref{d25.2} in the case $c_1 \leq 0 < c_2$ and we therefore omit it.

We will construct a decomposable FBM $X$, using the notation as in Definition \ref{d17.3}.
Let $\delta,\phi$ be strictly positive numbers such that $\delta<\phi \le 1/4$. Suppose that
$\{B^k_t, t\geq 0\}$, $k\in \Z  $, are independent Brownian motions and let $a_k\ge 1, \,k <0$ be numbers which we will specify later. 
Recall all the conditions that we imposed on $c_1, c_2, c_1''$, etc. in this part of the proof.
Let $T_k \equiv 1$ for $k\in \N_0$ and for $-k \in \N$ define
\begin{align}\label{d29.4}
T_{k} &:= 
\inf\{t \geq a_{k+1}:
B^{k}(0)-B^{k}(t) \in [c_1 \sqrt{t} , (c_1+\delta)\sqrt{t}]  \\
&\text{  and  }
B^{k}(t-s)-B^{k}(t) \in \left[c_1'' \sqrt{s},\,c_2 \sqrt{s}\land\left(c_1'' \sqrt{s} + \phi\right)\right]\  \forall s \in [0,a_{k+1}]\label{d30.2} \\
&\text{  and  }
B^{k}(t-s)-B^{k}(t) \in \left[c_1 \sqrt{s},\,c'_2 \sqrt{s} \right]\  \forall s \in [a_{k+1},t]\},\label{d30.3}
\end{align}
and note that $T_{k} < \infty$ a.s., by \eqref{d27.5}. By construction, the associated FBM $X$ satisfies $X_s \geq c_1\sqrt{|s|}$ for all $s \le 0$.

It remains to show that if the $a_k$ are suitably defined then we also have $\limsup_{t \to -\infty} X_t/\sqrt{|t|}\le c_2$ almost surely.

We will define $a_k$ for $k\leq 0$ inductively starting with $a_0=1$.
Suppose that the $a_{j}$'s have been defined for all $j> k$ for some $k< 0$. This determines $S_k$ and $X_t, \, t \ge S_k$. 
Let $V_k>0$ and $-\infty< R_k <-1$ 
be such that $\P(|X_{S_{k}}| \geq V_k \  \text{  or  }\ S_{k} <R_k) \leq 2^{k}$. 
Then we fix $ a_k\in(-R_k,\infty)$ such that for all $t_1\in[R_k, 0]$
and $t\leq t_1 -a_k$,  
\begin{align}\label{d29.3}
V_k + c'_2\sqrt{|t-t_1|} \leq c_2 \sqrt{|t|}\qquad \mbox{ and }\qquad V_k+ (c_1+\delta) \sqrt{|t-t_1|}
 \leq (c_1''+\phi)\sqrt{|t|}. 
\end{align}

Since $\sum_{k\leq 1} 2^{k} < \infty$, we see that, a.s., there exists a (random) $k_*\leq -1$ such that $|X_{S_{k}}| \leq V_k$ and 
$S_{k} \in [R_k, -1]$ for all $k\leq k_*$. 
If $X_{S_k}\le V_k$ and $S_k \geq R_k$ for some $-k \in \N$, then it follows from \eqref{d29.3}  and \eqref{d30.3}  that for $t \in [S_{k-1},S_k-a_k]$, we have
\begin{align}\label{29.5}
X_t \leq X_{S_k} + c'_2 \sqrt{S_k-t} \leq V_k +  c'_2 \sqrt{S_k-t}  \leq  c_2  \sqrt{|t|}. 
\end{align}
Further, if $X_{S_k}\le V_k$ and $S_k \geq R_k$  for some $-k \in \N$, then it follows from \eqref{d29.3} and \eqref{d29.4} that 
\begin{align}\label{ganzneu}
X_{S_{k-1}} \leq X_{S_k} + (c_1+\delta)\sqrt{S_k-S_{k-1} } \le V_k+ (c_1+\delta)\sqrt{S_k-S_{k-1} }\leq  (c_1''+\phi)\sqrt{|S_{k-1}| }  
\end{align}
and, for $t \in [S_{k-1}-a_{k-1},S_{k-1}]$, using \eqref{ganzneu}, \eqref{d30.2}, $c_1''\le 1, \,c_2 \ge 2$, the elementary inequality 
$\frac 32\sqrt{a}+\sqrt{b}\le 2 \sqrt{a+b}$ for $a,b \ge 0$       and $\phi \leq 1/4 \leq \sqrt{|S_{k-1}|}/4$      we have
\begin{align*}
X_t &\le X_{S_{k-1}}+(c_2\sqrt{S_{k-1}-t})\wedge(c_1''\sqrt{S_{k-1}-t}+\phi)\\
&\le (c_1''+\phi+\frac 14) \sqrt{|S_{k-1}|} +c_1'' \sqrt{S_{k-1}-t}       \leq c_2 \sqrt{|t|}. 
\end{align*}

Thus we have shown that $X_t \leq c_2\sqrt{|t|}$ holds for all 
$t\in (-\infty, S_{k_*-1}]$.
This completes the proof of part (ii) in the case $\lambda_0( c_1,c_2)<1$.

It remains to prove part (ii) in the case $0<c_1<c_2 \le \infty$ when $\lambda(c_1,c_2)=1$. In this case we proceed as above except that we replace $c_1$ in the definition 
of $T_k$ by $c_{1,k}$ such that $c_{1,k}$ approaches $c_1$ from below as $k \to -\infty$. This requires to let also $c_1''$ and $c_2'$ (but not $c_2$) depend on $k$. We leave the details 
to the reader.

This completes the proof of the theorem.
\end{proof}

In a particular case, we can construct an FBM which always lies above a parabolic boundary, and which is even {\em strongly} decomposable. 

\begin{proposition}
For each $\ve>0$, there exists a strongly decomposable FBM $X$ such that
$$
\inf_{t <0} {X_t}/{\sqrt{|t|}} \ge 1-\ve, \mbox{ a.s.}
$$
\end{proposition} 

\begin{proof}
Fix an arbitrarily small $\eps >0$.
Let $B$ be standard Brownian motion and
$$
T:=\inf\{t \ge 1: B(s) \ge B(t) + (1-\eps)\sqrt{t-s} \mbox{ for all } s \in [0,t]\}.
$$
Since $\lambda_0(1,\infty)=1$,
it follows from \eqref{d25.2} that
$\P(T< \infty)=1$. 
Let $X$ be the strongly decomposable FBM based on i.i.d. sequence
$(B^k,\, T_k)$, with elements distributed
as $(B,T)$ above and $X_0=0$. 
For $S_{n-1}\le t\le S_n$, $-n \in \N_0$, a.s.,
\begin{eqnarray*}
X_t&\ge& X_{S_n}+(1-\ve)\sqrt{S_n-t}\\
&\ge& X_{S_{n+1}} +(1-\ve)\sqrt{S_{n+1}-S_n}+ (1-\ve)\sqrt{S_n-t}\\
&\dots \\
&\ge&(1-\ve)\left( \sum_{k=n}^{-1} \sqrt{ S_{k+1}-S_k} +\sqrt{S_n-t} \right) \ge (1-\ve)\sqrt{-t}.
\end{eqnarray*}
\end{proof}

\section{A sufficient condition for FBM to be 2BM}\label{sec:2BM}

Two sided Brownian motion (2BM) is the most generic example of FBM but it is far from being a unique example of FBM, as the previous sections show. It is natural to ask what extra assumptions 
on an FBM make it necessarily 2BM.
We will present a sufficient condition for this to be true. We will also show that some other ``similar'' conditions fail to force an FBM to be 2BM.

Recall the notation used in  Definition \ref{d17.3}.

\begin{theorem}\label{d20.1}
If $X$ is strongly decomposable and $\E T_k < \infty$
then $X$ is 2BM.
\end{theorem}

\begin{proof} 
Assume that $X$ is strongly decomposable and $\E T_k < \infty$ for the $T_k$'s introduced in  Definition \ref{d17.3}.
We will assume without loss of generality that $X_0=S_0=0$.

According to \cite[Lemma 11.7]{Kall} (see also Theorem 11.4 in \cite{Kall} or Sections 4.1-4.2 and 8.1-8.2 in \cite{Thor}), there exists a random variable $\Theta$ such that the distribution of   
$\{S^*_n , n \in \Z\}:=\{S_n - \Theta, n \in \Z\}$ is stationary. 
Moreover, we can and will choose $\Theta$ so that it may depend on $\{S_n\}_{n\in \Z}$ but does not depend on $\{X_t\}_{t\in\R}$ in any other way.
It will suffice to show that the distribution of $\{X^*_t, t\in\R\} := \{X_{t+\Theta} - X_\Theta, t\in\R\}$ is 2BM(0).

Suppose that $a>0$, let $U_a$ be a uniform random variable on $[a, 2a]$, independent of $X$, and let $\{S^a_n , n \in \Z\}=\{S_n - U_a, n \in \Z\}$. Then it follows from \cite[Thm. 11.8 (i)]{Kall} that the
distributions of $\{S^a_n , n \in \Z\}$ converge to the distribution of $\{S^*_n , n \in \Z\}$ in the total variation norm, as $a\to\infty$. 
Let $X^a_t = X_{t+U_a} - X_{U_a}$ for $ t\in\R$.

The conditional distribution of $X$ given $\{S_n, n\in \Z\}$ can be described as follows.
Suppose that $\bs=\{s_n, n\in \Z\}$ is a deterministic sequence of real numbers such that $s_n < s_{n+1}$ for all $n$, $\lim_{n\to - \infty} s_n = -\infty$ and $\lim_{n\to  \infty} s_n = \infty$. 
Let $Q$ be the distribution of a pair $(T_k, B^k)$ used in the construction of the strongly decomposable process $X$ (note that $Q$ does not depend on $k$).
Let $Q_t$ be the distribution $Q$ conditioned by $\{T_k = t\}$ and let $\wt Q_t$ be the distribution of the second element in the pair (stochastic process) under $Q_t$, stopped at $t$. Let $\{\wt B^n, n\in \Z\}$ be independent processes, such that the distribution of $\wt B^n$ is $\wt Q_{s_{n+1} - s_n}$ for all $n$. Let $\wt X$ be the unique continuous process such that $\wt X_{t+s_n} - \wt X_{s_n} = \wt B^n _t $ for all $t\in [0, s_{n+1}-s_n)$ and all $n\in\Z$. Let $\cD(\bs)$ denote the distribution of $\wt X$. Then the distribution of $X$ is $\cD(\{S_n, n\in \Z\})$. Similarly, the distributions of $X^*$ and $X^a$ are $\cD_* :=\cD(\{S^*_n, n\in \Z\})$ and $\cD_a :=\cD(\{S^a_n, n\in \Z\})$, resp.
In other words, the conditional distributions of $X^*$ and $X^a$ given $\{S^*_n, n\in \Z\}$ and $\{S^a_n, n\in \Z\}$, resp., are identical.
Since the distribution of $\{S^a_n , n \in \Z\}$ converges to the distribution of $\{S^*_n , n \in \Z\}$ in the total variation norm, $\cD_a$ converge to $\cD_*$ in the total variation norm, as $a\to\infty$.

It follows from the definition of a decomposable FBM that $\{X_t, t\geq 0\}$ is standard Brownian motion. Hence, for every fixed $s\in[a,2a]$, the distribution of $\{X_{t+s} - X_s, -a \leq t \leq a\}$ is that of 2BM(0) with time restricted to the interval $[-a,a]$. Since $U_a$ is independent of $X$, the distribution of $\{X^a_t, -a \leq t \leq a\}$ is also that of 2BM(0) restricted to $[-a,a]$. This in turn implies that for any fixed $b>0$ and all $a\geq b$, the distribution of $\{X^a_t, -b \leq t \leq b\}$ is that of 2BM(0) restricted to $[-b,b]$.
In other words, for any fixed $b>0$ and all $a\geq b$, the distribution $\cD_a$ restricted to $[-b,b]$ is that of 2BM(0). Since $\cD_a$ converges to $\cD_*$ in the total variation norm, as $a\to\infty$, we conclude that for any fixed $b>0$, the distribution $\cD_*$ restricted to $[-b,b]$ is that of 2BM(0). The constant $b>0$ is arbitrarily large so the distribution $\cD_*$ is that of 2BM(0) on the whole real line.
\end{proof}

We will show that the result in Theorem \ref{d20.1} is optimal, in a sense. First, we will show that 
the conclusion of Theorem \ref{d20.1} does not necessarily hold if the assumption $\E T_k <\infty$ is replaced by the condition $\E T_k^\alpha <\infty$ for some $\alpha\in (0,1)$. Next, we will show that if $T_k$'s are not i.i.d. then the condition $\sup_k \E T_k <\infty$ does not guarantee that the corresponding FBM is 2BM. Moreover, even if $\sup_k \E T_k^\alpha <\infty$ for some $\alpha < \infty$, the FBM is not necessarily 2BM.

\begin{theorem}
For any $\alpha\in (0,1)$,
there exists a strongly decomposable FBM $X$ satisfying $\E T_k^\alpha <\infty$ which is not a BBM.
\end{theorem} 

\begin{proof}
Fix any $\alpha\in (0,1)$ and find
$c_1 < 1$ such that $\lambda_0(c_1,\infty) = (1+\alpha)/2$. Let 
\begin{align*}
T_k = \inf\{t\geq 0: B^k_t \leq -1 + c_1\sqrt{t}\}.
\end{align*}
It follows from \cite[Lem.~10(b)]{Perkins} that
\begin{align*}
\P(T_k \geq t) =
\P\{B^k_u \geq -1+ c_1\sqrt{u}\ \ \forall 0 \le u \le t\}
\leq K  t^{-\lambda_0(c_1,\infty)}
= K  t^{-(1+\alpha)/2},
\end{align*} 
where $K >0$. This implies that 
$\E T_k^\alpha <\infty$.

Suppose that there is a random variable $S$ such that $\{X_{S-t} - X_S, t\geq 0\}$ is Brownian motion. We will show that this assumption leads to a contradiction.

Recall that $X_0=0$. For $n\geq 1$ and $t\geq 0$, let
\begin{align*}
X^{S,n}_t = \frac1{\sqrt{n}} (X_{S-nt} - X_S),
\qquad
X^{0,n}_t = \frac1{\sqrt{n}} X_{-nt} ,
\qquad
X^n_t = -\frac1{\sqrt{n}} X_{nt} .
\end{align*}

It is easy to see that for any random variable $S$ and continuous process $X$, the sequence of processes $\{X^{S,n}_t, t\geq 0\}$ converges to Brownian motion in the Skorokhod topology if and only if $\{X^{0,n}_t, t\geq 0\}$ converges to Brownian motion.
We have assumed that $\{X_{S-t} - X_S, t\geq 0\}$ is Brownian motion so $\{X^{S,n}_t, t\geq 0\}$ is Brownian motion for every $n$. Hence, to complete the proof, it will suffice to show that $\{X^{0,n}_t, t\geq 0\}$ does not converge to Brownian motion.

Recall $S_k$'s from Definition \ref{d17.3} and let $S^0_k = - S_{-k}$.
For $t\in [S^0_k/n, S^0_{k+1}/n]$, $k\geq 0$, let
\begin{align*}
X^{+,n}_{t}= - X^{0,n}_{-t+S^0_k/n+S^0_{k+1}/n}
+X^{0,n}_{S^0_{k+1}/n}
+X^{0,n}_{S^0_k/n}.
\end{align*}
The process $X^{+,n}$ is obtained from the process $X^{0,n}$ by rotating every piece of the trajectory 
between $S^0_k/n $ and $ S^0_{k+1}/n$ by 180 degrees and matching the endpoints
of the rotated path with the original locations of the endpoints. It is easy to see that the distribution of $\{X^{+,n}_t, t\geq 0\}$ is the same as that of
$\{X^n_t, t\geq 0\}$. Hence, $\{X^{+,n}_t, t\geq 0\}$ is Brownian motion.
It will be enough to show that 
 $\{X^{0,n}_t, t\geq 0\}$ and $\{X^{+,n}_t, t\geq 0\}$
do not converge to the same limit, in distribution.

Assume to the contrary that $\{X^{0,n}_t, t\geq 0\}$ and $\{X^{+,n}_t, t\geq 0\}$
converge to the same limit, in distribution. The limit must be
Brownian motion. Since each sequence is tight, the sequence of pairs
$\{(X^{0,n}_t, X^{+,n}_t), t\geq 0\}$ is also tight.
Therefore, it contains a convergent subsequence.
By abuse of notation, we will assume that the whole
sequence converges in distribution. Let the weak limit be called
$\{(X^{0,\infty}_t, X^{+,\infty}_t), t\geq 0\}$.

Note that since $B^k_t \geq -1$ for $t\in[0, T_k]$, we have
$X^{+,n}_t - X^{0,n}_t \leq \frac1{\sqrt{n}} $ for $t\geq 0$, a.s.
This implies that 
$X^{+,\infty}_t - X^{0,\infty}_t \leq 0 $ for $t\geq 0$, a.s.

It follows from \cite[Lem.~10]{Perkins} that
\begin{align*}
\P(T_k \geq t) 
\sim K_1  t^{-\lambda_0(c_1,\infty)}
= K_1  t^{-(1+\alpha)/2},
\end{align*} 
where $K_1 >0$. Since $\alpha\in(0,1)$, 
standards results for sums of heavy tailed random variables (see, e.g., 
\cite[Thm.~5.1]{Dar} or
\cite[Sect.~10.5, p.~150]{KorSin}) show that the size of $\max_{1\leq k \leq n} T_k$ is comparable to $S_n$ with positive probability. More precisely, for some $p_1 >0$ and $\beta\in(0,1/2)$,
for every $n\geq 1$, with probability greater than $p_1$, there exists
$k$ such that the following event holds,
$A=\{0\leq S^0_k/n< S^0_k/n +\beta < S^0_{k+1}/n \leq 1\}$. If $A$ holds then
\begin{align*}
X^{0,n}(S^0_k/n +\beta/2) - X^{+,n}(S^0_k/n +\beta/2) \geq 
c_1 \sqrt{\beta/2} - c_1 \sqrt{\beta}/2 - \frac1{\sqrt{n}}.
\end{align*}
This and the assumption that the limits $X^{0,\infty}$ and $ X^{+,\infty}$ are continuous processes imply that with probability greater than $p_1$, there exists
$t\in[0,1]$ such that 
\begin{align*}
X^{0,\infty}_t - X^{+,\infty}_t \geq 
c_1 \sqrt{\beta/2} - c_1 \sqrt{\beta}/2 .
\end{align*}
This and the fact that 
$X^{+,\infty}_t - X^{0,\infty}_t \leq 0 $ for $t\geq 0$, a.s.,
imply that the processes  $\{X^{0,n}_t, t\geq 0\}$ and $\{X^{+,n}_t, t\geq 0\}$
do not converge to the same process with continuous paths, in distribution.
\end{proof}

\begin{theorem}
For any $\alpha\in (0,\infty)$,
there exists a  decomposable FBM $X$ satisfying $\sup_i \E T_i^\alpha <\infty$ which is not a BBM.
\end{theorem} 

\begin{proof}
Assume that $p_j \in (0,1)$, $k_j \in \N$, and $c_j>0$ for each $j \in \N$ (we will specify the values of these parameters later in the proof). 
For $i \in \N$ let $j(i)$ be the unique integer $j$ satisfying 
$\sum_{m=1}^{j-1} k_m +1 \le i \le \sum_{m=1}^{j} k_m$. For each $i\in \N$ toss a coin which comes up heads with probability $p_{j(i)}$ (independently of 
everything else) and define $T_{-i}:= \inf\{t \ge 1: B^{-i}_t-B^{-i}_{t-1}=c_{j(i)}\}$ if coin $i$ comes up heads and $T_{-i}=0$ otherwise.  Further, for $i \in \N_0$, we define $T_i \equiv 1$. 
For $c>0$ and a Brownian motion $B$, let $\lambda(c):= \E (\inf\{t \ge 1: B_t-B_{t-1}=c\})^\alpha$. 
It is easy to see that $\inf\{t \ge 1: B_t-B_{t-1}=c\}$ is stochastically majorized by a constant plus an exponential random variable 
so $\lambda(c)$ is finite for every $c<\infty$ and $\alpha \in (0,\infty) $. 

We now define the numbers $p_j,\,k_j,\,c_j$ recursively starting with $c_1=1$. Given the numbers $c_1,...,c_m$, $p_1,...,p_{m-1}$, and $k_1,...,k_{m-1}$, we
define $p_{m}:=1/\lambda(c_m)$. This implies that $\E T_m^\alpha=1$ for all $m\in \Z$. Let $k_m:=\lceil 1/p_m \rceil$. This implies that $\sum_{j\in\N} k_j p_j=\infty$ and therefore guarantees that   infinitely many of the $T_{-i}$, $i\in\N$, are at least 1. Let $u_m$ be a positive number such that
$$
\P\left( \sum_{i=1}^{k_1+...+k_m} T_{-i}+1 \ge u_m\right)\le 2^{-m}.
$$ 
Then, choose $c_{m+1}$ so large that for a Brownian motion $B$ we have
$$
\P( \inf\{t \ge 1: B_t-B_{t-1}=c_{m+1}\} < u_m + m ) \le 1/2. 
$$
This completes the definition of $p_j$'s, $k_j$'s and $c_j$'s.

Let $X$ be the decomposable FBM  associated to the sequence $(B^i,T_i)$. Assume that $X$ is a BBM. We will show that this assumption leads to a contradiction. Suppose that $S$ is a random variable such that 
$\{Y(t):=X(S-t)-X(S), t \ge 0\}$ is a Brownian motion.  We have
\begin{align*}
\P(&\inf\{t \ge 1: Y(t)-Y(t-1)=c_{m+1}\} \ge u_m + m)\\
&\le  \P(S\ge m)
+\P\left(S\le -\sum_{i=1}^{k_1+...+k_m} T_{-i}\right)
+\P\left(\sum_{i=1}^{k_1+...+k_m} T_{-i}+1 \ge u_m\right).
\end{align*}
Note that each probability on the right hand side
converges to 0 as $m \to \infty$. On the other hand,
\begin{align*}
&\P(\inf\{t \ge 1: B_t-B_{t-1}=c_{m+1}\} \ge u_m + m) \\
&\quad =1-\P(\inf\{t \ge 1: B_t-B_{t-1}=c_{m+1}\} < u_m + m)\ge 1/2,
\end{align*}
for each $m$, so $Y$ and $B$ cannot have the same law and the proof of the theorem is complete.
\end{proof}

\section{A process that is an FBM and BBM but not a 2BM}\label{sec:FBMBBM}

``Most'' local path properties of every FBM are the same as those of standard Brownian motion. For example, FBM paths are continuous, non-differentiable and satisfy the local law of the iterated logarithm at almost all (with respect to Lebesgue measure) times. We said ``most'' properties because there are some clear exceptions, for example, $X_t>0$ for $t\in[-\eps, 0)$, for every $\eps>0$, if $X$ is 
constructed as in Example \ref{d17.1}. Needless to say, standard Brownian motion does not have this property. However, this exception is clearly an artifact of the construction given in Example \ref{d17.1} and does not characterize a ``typical'' local behavior of the paths of $X$ in that example.

The definition of FBM implies that the global path properties of FBM, such as the global law of the iterated logarithm, are identical to those of standard Brownian motion in the forward time direction. If we now assume that a process is both FBM and BBM, then this process has the same global path properties as standard Brownian motion in the forward and backward time directions. Hence, such a process has the same (or very similar) local and global path properties as 2BM. 
It is tempting to conjecture that this process is a 2BM
because it is hard to guess in what way this process might be different from 2BM. Nevertheless, it turns out that there exists a  
process that is FBM and BBM but not 2BM. The reason why this is possible is, roughly speaking, that the increments of this process are heavily correlated on scales that are ``invisible'' if we observe the process from the viewpoints set at some random times.

The presentation of our construction will be discrete in nature. 
See Definition \ref{d18.7} for the definitions of FRW, BRW and 2RW.

\begin{theorem}\label{d22.2}
There exists a process $X$ which is FBM and BBM but not 2BM.
\end{theorem}

\begin{proof}
We will first construct a process $\{V_k,\, k \in \Z\}$, taking values in $\{-1,1\}$, which is the increment sequence of a process which is FRW and BRW 
but not a 2RW.
The construction will be inductive.
At the $n$-th step, we will define the values of $V_k$ for $k\in [a_n, b_n]$, where $a_n$ and $b_n$ are random integers satisfying $a_{n+1} < a_n < 0 < b_n < b_{n+1}$ for all $n\in \N$, a.s. 

We will call a sequence of random variables {\em coin tosses} if they are i.i.d., taking values 1 and $-1$ with equal probabilities.

For $n=1$, we take $a_1 = -1$, $b_1=1$ and we let $V_k$, $-1\leq k\leq 1$, be coin tosses.

Suppose that $[a_n,b_n]$ and 
$\{V_k,k\in [a_n, b_n]\}$ have been defined.

Let $c_n\in\N$ be a constant so large that 
\begin{align}\label{j6.3}
\P(|a_n| \lor b_n \geq c_n) < 1/n^2.
\end{align}
Let $d_n\in\N$ be so large that $(4c_n+1)^2 2^{-d_n} < 1/n^2$.

Let $\{V_k,k\in [ b_n+1, b_n + d_n]\}$ be coin tosses independent of $\{V_k,k\in [a_n, b_n]\}$ and let $V_k = V_{k - a_n + d_n + b_n +1}$ for $k \in [a_n- d_n, a_n -1]$. If we set $a'_n = a_n -d_n$ and $b'_n = b_n + d_n$ then we see that $\{V_k,k\in [ a'_n, b'_n ]\}$ has been defined.

Let $\{U^{n-}_k, k < a'_n\}$ and $\{U^{n+}_k, k > b'_n\}$ be two sequences of coin tosses independent from each other and jointly independent of $\{V_k,k\in [ a'_n, b'_n ]\}$. 
Let $a_{n+1} $ be the largest integer of the form $a_{n+1} = j (b'_n - a'_n +1) + a'_n$ for some $j < 0$, with the property that $ U^{n-}_{k +a_{n+1} - a'_n}=V_k$ for all 
$k\in [ a'_n, b'_n ]$. Since $\{U^{n-}_k, k < a'_n\}$ are coin tosses, it is easy to see that such an integer $a_{n+1}$ exists. By analogy, we define 
$b_{n+1} $ as the smallest integer of the form $b_{n+1}= j (b'_n - a'_n +1) + a'_n$ for some $j > 0$, such that $ U^{n+}_{k +b_{n+1} - b'_n}=V_k$ for all 
$k\in [ a'_n, b'_n ]$.

We let $V_k = U^{n-}$ for $k\in [a_{n+1}, a'_n -1]$ and 
$V_k = U^{n+}$ for $k\in [ b'_n +1, b_{n+1}]$.
We have thus defined $[a_{n+1},b_{n+1}]$ and 
$\{V_k,k\in [a_{n+1},b_{n+1}]\}$. This completes the inductive step and the definition of $\{V_k, k\in \Z\}$.

Let $Z_0=0$ and $Z_{k+1} - Z_k = V_k$ for $k\in \Z$.
We will argue that $Z$ is FRW and BRW but not 2RW.

Fix any $m\in\N$ and
a deterministic sequence $\bs\in \{-1,1\}^m$. Let $Q_\bs$ be the distribution of the sequence of $m$ coin tosses conditioned not to be equal to $\bs$. The probability that a sequence of $m$ coin tosses is not equal to $\bs$ is $p_m := 1 - 2^{-m}$. Let $\alpha_1$ and $\alpha_2$ be independent geometric random variables with parameter $p_m$, that is $\P(\alpha_i = k) = p_m^ k (1-p_m)$ for $i=1,2$ and $k\in \N_0$. Let $\{Y^{i,j}_n, n\in[1,m]\}_{j=1,\dots,\alpha_i}$, $i=1,2$, be i.i.d. sequences with distribution $Q_\bs$, independent of each other and of $\alpha_1$ and $\alpha_2$. If $\alpha_1>0$, let $R_{(j-1)m+n}  = Y^{1,j}_n$ for $j=1,\dots,\alpha_1$ and $n=1,\dots, m$. Let $\{R_n, n= \alpha_1 m +1, \dots, (\alpha_1+1) m\} = \bs$.
If $\alpha_2>0$, let $R_{(\alpha_1+ j)m+n}  = Y^{2,j}_n$ for $j=1,\dots,\alpha_2$ and $n=1,\dots, m$. Let $\{R_n, n= (\alpha_1 +\alpha_2 +1) m +1, \dots, (\alpha_1+\alpha_2+2) m\} = \bs$.
Let $\{R_n, n \geq (\alpha_1+\alpha_2+2) m +1\}$ be a sequence of coin tosses independent of $\{R_n, n= 1, \dots, (\alpha_1+\alpha_2+2) m\} $. It is elementary to see that $\{R_n, n\geq 1\}$ is a sequence of coin tosses.

Since the distribution of $\{R_n, n\geq 1\}$ does not depend on $m$ or $\bs$, we see that if $m\in \N$ and $\bs\in \{-1,1\}^m$ are chosen in an arbitrary random way, the distribution of $\{R_n, n\geq 1\}$ is still that of a sequence of coin tosses.

Let $S_{-n}=a_{n+1} + b'_n -a'_n +1 $ for $n\in \N$. We will argue that $\{V_k, k\geq S_{-n}\}$ is a sequence of coin tosses. If we take $m = b'_n -a'_n +1$ and $\bs = \{V_k,k\in [ a'_n, b'_n ]\}$ then it follows from our constructions of $\{V_k, k\in \Z\}$ and $\{R_n, n\geq 1\}$ that the distribution of $\{V_k, k\geq S_{-n}\}$ is the same as that of $\{R_n, n\geq 1\}$ and hence it is the distribution of a sequence of coin tosses. Since $S_{-n}\to -\infty$, we conclude that $Z$ is FRW. The process $Z$ is BRW by the symmetry of our construction.

We will now assume that $Z$ is 2RW and we will show that this leads to a contradiction.
Let $S$ be such that 
$\{Z_{S+k}-Z_S,\,k\in \N_0\}$ and $\{Z_{S-k}-Z_S,\,k\in \N_0\}$ are independent simple symmetric random walks. 
If $W_k = Z_{S+k+1} - Z_{S+k}$ for $k\in \Z$ then $\{W_k, k\in \Z\}$ is a sequence of coin tosses. 
For an arbitrarily large $m$, we can find $n>m$ so large that $\P(|S| \geq c_n) < 1/m^2$.  
Recall that $\P(|a_n| \lor b_n \geq c_n) < 1/n^2$. Hence, 
\begin{align}\label{d21.1}
\P(\{|S| \geq c_n\}\cup\{|a_n| \lor b_n \geq c_n\}) < 1/m^2 + 1/n^2 \leq 2/m^2.
\end{align}
If $a^*$ and $ b^*$ are fixed integers such that $a^*< b^*$ then the probability that
$W_k = W_{k - a^* + d_n + b^* +1}$ for $k \in [a^*- d_n, a^* -1]$ is $2^{-d_n}$.
The probability that there exist integers
$a^*,b^*\in[-2c_n,2c_n]$ such that
$a^*< b^*$ and $W_k = W_{k - a^* + d_n + b^* +1}$ for $k \in [a^*- d_n, a^* -1]$ is bounded above by $(4c_n+1)^2 2^{-d_n} < 1/n^2 < 1/m^2$. The series $\sum_m 3/m^2$ is summable 
so the last estimate, \eqref{d21.1} and the Borel-Cantelli Lemma imply that there exist infinitely many $n$ such that
$|S| < c_n$, $|a_n| \lor b_n < c_n$ and there are no $a^*,b^*\in[-2c_n,2c_n]$ such that
$a^*< b^*$ and $W_k = W_{k - a^* + d_n + b^* +1}$ for $k \in [a^*- d_n, a^* -1]$. 
This contradicts the fact that for every $n>1$,
$V_k = V_{k - a_n + d_n + b_n +1}$ for $k \in [a_n- d_n, a_n -1]$. 

Let $X$ be defined in terms of $Z$ as in Remark \ref{d18.10} (i). We have indicated in that remark that the fact that $Z$ is FRW and BRW implies that $X$ is FBM and BBM. It remains to show that $X$ is not 2BM. 

Let $\eps_n>0$ be so small that for standard Brownian motion $B$ and any $x\in \R$, 
\begin{align}\label{j6.2}
\P(\exists t\in [1-\eps_n, 1+\eps_n]: 
|B_t -x|\leq 2\eps_n) \leq 1/n^2.
\end{align}
We can find $s_n \in (0,\eps_n)$ so small that
\begin{align}\label{j6.9}
\P(\exists s,t\in [1-s_n, 1+s_n]: 
|B_t -B_s|\geq \eps_n) \leq 1/n^2.
\end{align}
This and \eqref{j6.2} imply that if $B$ and $B'$ are independent Brownian motions then
\begin{align}\label{j5.1}
\P(\exists s,t\in [1-s_n, 1+s_n]: 
|B_t-B'_s| \leq \eps_n) \leq 2/n^2.
\end{align}

Note that in the first part of our proof, we can take $d_n$ arbitrarily large relative to $c_n$. 
Hence, we can and will assume without loss of generality that 
\begin{align}\label{j6.1}
\frac{d_n}{d_n+c_n} \geq 1- s_n/4.
\end{align}
We make $d_n$ larger, relative to $c_n$, if necessary, so that
\begin{align}\label{j6.7}
4 d_n^{-1/2} c_n \leq \eps_n/2.
\end{align}

The random variables $M_{j+1} - M_j$ defined in Remark \ref{d18.10} (i) are i.i.d. They represent the time  Brownian motion starting from 0 takes to hit 1 or $-1$. It is well known that these random variables have mean 1 and exponential tails. This, \eqref{j6.3}, \eqref{j6.1} and the law of large numbers imply that 
\begin{align}\label{j6.4}
\P\left(\left|\frac{M_{a_n-d_n} }{d_n} -1\right| \geq s_n/2\right) \leq 2/n^2,
\qquad
\P\left(\left|\frac{M_{b_n+d_n} }{d_n} -1\right| \geq s_n/2\right) \leq 2/n^2.   
\end{align}

Suppose that a random time $S$ is such that 
$\{X_{S+t}-X_S,\,t \ge 0\}$ and  $\{X_{S-t}-X_S,\,t \ge 0\}$ are independent standard Brownian motions.
We can make $d_n$'s larger, if necessary, so that the products $d_n s_n$ are so large that 
for $k\in \N$ we can find $n_k> k\lor n_{k-1}$ so large that 
\begin{align*}
\P\left(\left|S/d_{n_k} \right| \geq s_{n_k}/4\right) \leq 1/k^2.   
\end{align*}
This and \eqref{j6.4} yield
\begin{align}\label{j6.5}
&\P\left(\left|\frac{M_{a_{n_k}-d_{n_k}} -S}{d_{n_k}} -1\right| \geq s_{n_k}/2\right) \leq 2/{n_k}^2 + 1/k^2 \leq 3/k^2,\\
&\P\left(\left|\frac{M_{b_{n_k}+d_{n_k}} -S}{d_{n_k}} -1\right| \geq s_{n_k}/2\right) \leq  3/k^2.  \nonumber 
\end{align}
Let 
\begin{align*}
G_k = \left\{
\left|\frac{M_{a_{n_k}-d_{n_k}} -S}{d_{n_k}} -1\right| \geq s_{n_k}/2, \left|\frac{M_{b_{n_k}+d_{n_k}} -S}{d_{n_k}} -1\right| \geq s_{n_k}/2\right\}.
\end{align*}
It follows from \eqref{j6.5}, summability of $\sum_{k\in\N} 3/k^2$ and Borel-Cantelli Lemma that only a finite number of events $G_k$ occur.

For any $k \geq 1$, the processes 
\begin{align*}
&\{B^{(k)}_t := d_{n_k}^{-1/2}(X(S+t d_{n_k})-X(S)),\,t \ge 0\},\\
&\{B^{[k]}_t := d_{n_k}^{-1/2}(X(S-t d_{n_k})-X(S)),\,t \ge 0\}
\end{align*}
are independent Brownian motions.
Let 
\begin{align*}
F_k = \{\exists s,t\in [1-s_{n_k}, 1+s_{n_k}]: 
|B^{(k)}_t-B^{[k]}_s| \leq \eps_{n_k}\}.
\end{align*}
By \eqref{j5.1}, $\P(F_k) \leq 2/n_k^2 < 2/k^2$.
Since $\sum_{k\in \N} 2/k^2 < \infty$, only a finite number of events $F_k$ occur.

It follows from \eqref{j6.3} that only a finite number of events $\{|a_n| \lor b_n \geq c_n\}$ occur. Assuming that $|a_n| \lor b_n \leq c_n$, 
\begin{align*}
|Z_{a_n-d_n} - Z_{b_n-d_n}| \leq 2c_n.
\end{align*}
It follows that, for sufficiently large $n$,
\begin{align*}
|X(M_{a_n-d_n}) - X(M_{b_n-d_n})| \leq 2c_n,
\end{align*}
and, therefore, for all sufficiently large $k$,
\begin{align}\label{j6.8}
|B^{(k)}((M_{a_{n_k}-d_{n_k}}-S)/d_{n_k})-B^{[k]}((M_{b_{n_k}+d_{n_k}}-S)/d_{n_k})|
\leq d_{n_k}^{-1/2} 2 c_{n_k} \leq \eps_{n_k} /2, 
\end{align}
where the last inequality holds by \eqref{j6.7}.
Recall that only a finite number of events $G_k$ occur. If $G_k$ does not hold then, because of \eqref{j6.8}, $F_k$ holds with
\begin{align*}
t=(M_{a_{n_k}-d_{n_k}}-S)/d_{n_k},
\qquad s=(M_{b_{n_k}+d_{n_k}}-S)/d_{n_k}. 
\end{align*}
This contradicts the fact that only a finite number of events $F_k$ hold.
\end{proof}

\begin{proposition}\label{d22.1}
(i) There exists an FBM $X$ such that there is no random time $T$ such that $\{X_{T+t}-X_T, t\geq 0\}$ and $\{X_{T-t}-X_T, t\geq 0\}$ are independent and $\{X_{T+t}-X_T, t\geq 0\}$ is standard Brownian motion.

(ii) There is an FBM $X$ that is not decomposable.
\end{proposition}

\begin{proof}
(i) First, we will show that for the FRW $Z$ constructed in Theorem \ref{d22.2}, there is no stopping time $S$ such that 
$\{Z_{S+k}-Z_S, k\in\N_0\}$ and $\{Z_{S-k}-Z_S, k\in\N_0\}$ are independent and $\{Z_{S+k}-Z_S, k\in\N_0\}$ is simple symmetric random walk. 
We will apply the same argument as
in the part of the proof of Theorem \ref{d22.2} showing that $Z$ is not 2RW.
We replace the paragraph in that proof containing \eqref{d21.1} with the following.

Let $S$ be such that 
$\{Z_{S+k}-Z_S, k\in\N_0\}$ and $\{Z_{S-k}-Z_S, k\in\N_0\}$ are independent and $\{Z_{S+k}-Z_S, k\in\N_0\}$ is simple symmetric random walk.
If $W_k = Z_{S+k+1} - Z_{S+k}$ for $k\in \Z$ then $\{W_k, k\in \Z\}$ is a sequence of coin tosses.
Let $c_0=1$ and recall that $a_{k+1} < a_k < 0 < b_k < b_{k+1}$ for all $k\in \N$. 
For $m\in\N$, we can find $n_m>m$ so large that $\P(|S| \geq c_{n_{m-1}}) < 1/m^2$ and $c_{n_{m-1}}
\leq |a_{n_m}| \land b_{n_m}$.  
Recall that $\P(|a_k| \lor b_k \geq c_k) < 1/k$. Hence, 
\begin{align}\label{d22.3}
\P(\{|S| \geq c_{n_{m-1}}\}\cup\{c_{n_{m-1}}
\leq |a_{n_m}| \land b_{n_m} \leq |a_{n_m}| \lor b_{n_m} \geq c_{n_m}\}) < 1/n_m^2
+ 1/m^2 \leq 2/m^2.
\end{align}
If $a^*$ and $ b^*$ are fixed integers such that $a^*<0< b^*$ then the probability that the events
$\{W_k = W_{k - a^* + d_{n_m} + b^* +1}\}$ hold for $k \in [a^*- d_{n_m}, a^* -1]$ is $2^{-d_{n_m}}$.
The probability that there exist integers
$a^*,b^*\in[-2c_{n_m},2c_{n_m}]$ such that
$a^*< b^*$ and $W_k = W_{k - a^* + d_{n_m} + b^* +1}$ for $k \in [a^*- d_{n_m}, a^* -1]$ is bounded above by $(4c_{n_m}+1)^2 2^{-d_{n_m}} < 1/{n_m}^2 < 1/m^2$. The series $\sum_m 3/m^2$ is summable 
so the last estimate, \eqref{d22.3} and the Borel-Cantelli Lemma imply that there exist infinitely many $n_m$ such that
$|S| < c_{n_m}$, $|a_{n_m}| \lor b_{n_m} < c_{n_m}$ and there are no $a^*,b^*\in[-2c_{n_m},2c_{n_m}]$ such that
$a^*<0< b^*$ and $W_k = W_{k - a^* + d_{n_m} + b^* +1}$ for $k \in [a^*- d_{n_m}, a^* -1]$. 
This contradicts the fact that for every $n>1$,
$V_k = V_{k - a_n + d_n + b_n +1}$ for $k \in [a_n- d_n, a_n -1]$. 

This completes the proof that for the FRW $Z$ constructed in Theorem \ref{d22.2}, there is no stopping time $S$ such that 
$\{Z_{S+k}-Z_S, k\in\N_0\}$ and $\{Z_{S-k}-Z_S, k\in\N_0\}$ are independent and $\{Z_{S+k}-Z_S, k\in\N_0\}$ is simple symmetric random walk. 

Suppose that there exists a random time $T$ such that $\{X_{T+t}-X_T, t\geq 0\}$ and $\{X_{T-t}-X_T, t\geq 0\}$ are independent and $\{X_{T+t}-X_T, t\geq 0\}$ is standard Brownian motion. Then we can proceed as in the proof
of Theorem \ref{d22.2}, staring with the paragraph containing \eqref{j6.2}. Note that for \eqref{j6.9}, we only need to know that the process $B$ is a.s. continuous (we do not have to assume that it is Brownian motion). The rest of the argument applies and thus we complete the proof of part (i) the proposition.

(ii) Suppose that the FBM $X$ considered in part (i) is decomposable. 
Then, in the notation of Definition \ref{d17.3}, there is a random variable $U$ such that $\{X_{U+t}-X_U, t\geq 0\}$ and $\{X_{U-t}-X_U, t\geq 0\}$ are independent and $\{X_{U+t}-X_U, t\geq 0\}$ is standard Brownian motion.
This contradicts part (i) so we conclude that FBM $X$ is not decomposable.
\end{proof}

\section{Open problems}\label{sec:open}

The following list is rather eclectic but we hope that the reader will find at least some of the problems intriguing.

\begin{problem}\label{pr1}
Assume that $X$ is a decomposable FBM which is also a BBM (we may or may not assume that the BBM is decomposable). Do these assumptions imply that $X$ is 2BM?
\end{problem}

\begin{problem}\label{pr2}
Assume that $X$ is an FBM and BBM and there exists a random time $T$ such that the processes $X_{T+t}-X_T,\,t \ge 0$ and $X_{T-t}-X_T,\,t\ge 0$ are independent. 
Do these assumptions imply that $X$ is 2BM?
\end{problem}

\begin{problem}\label{pr3}
Assume that $X$ is an FBM  and there exist random 
times $S$ and $T$ such that the processes $X_{T+t}-X_T,\,t \ge 0$ and $X_{T-t}-X_T,\,t\ge 0$ are independent 
and such that $X_{S-t}-X_S,\,t\ge 0$ is Brownian motion.  
Do these assumptions imply that $X$ is 2BM?
\end{problem}

Note that an affirmative answer to Problem \ref{pr1} implies the same for Problem \ref{pr2} and, similarly,
an affirmative answer to Problem \ref{pr2} implies the same for Problem \ref{pr3}.

\begin{problem}
Consider a decomposable FBM $X$ and assume that is is constructed from Brownian pieces of length 1, most of the time, but occasionally 
(more and more rarely as we move to the left) we insert ``Bessel'' pieces, i.e., we use the stopping times $T_k:=\inf\{t \ge 0: B^k_t=-1\}$.
Under which conditions (concerning the frequency of the Bessel pieces) is the resulting FBM a BBM (or 2-sided BM)? 
\end{problem}

\begin{problem}
Is it true that for any process $\{X_t,\,t\ge 0\}$ whose law is equivalent to BM, we can find some 
random piece which we can put in front of $X$ such that the new process is Brownian motion?
\end{problem}
 
\begin{remark}
Consider a strongly decomposable FBM $X$ with $\E T_k = \infty$. In this case 
$X$ may or may not be a 2BM. If, for example, $T_k$ is a positive random variable which is independent of $B_k$ and has an infinite expected value, 
then $X$ is clearly 2BM (even without 
shifting). Now let us assume that $\E T_k = \infty$ and that the process $\{S_n, n\in\Z\}$ is 
{\em identifiable} in the sense that
there exists a measurable function that maps $\{B_{T+t} - B_T, t\in \R\}$ onto $\{S_n-T, n\in\Z\}$
for an arbitrary random time $T$. 
If $X$ was 2BM, then this process would have to be stationary seen from the random time which turns $X$ into 2BM(0). But for a 
renewal process with infinite expected interarrival law there does not exist any shift which will make it stationary. 
\end{remark}

\begin{problem}
Does there exist a  decomposable FBM $X$ satisfying $\sup_k |T_k| <\infty$, a.s., which is not a BBM?
\end{problem}

\begin{problem}
Can one generalize Theorem \ref{d30.5} from parabolas to other space-time shapes?
\end{problem}

\begin{problem}
Analyze ``forward L\'evy processes''. In particular, find analogues of all theorems in this article for forward L\'evy processes.
\end{problem}

\section{Acknowledgments}

We are grateful to Leif Doering, Ander Holroyd, Olav Kallenberg, Haya Kaspi, Wilfrid Kendall, G\"unter Last, Peter M\"orters, Yuval Peres, Ed Perkins, Hermann Thorisson and Jon Wellner for
very helpful advice.

\bibliographystyle{abbrvnat}
\bibliography{ForwardBM}

\end{document}